\documentclass[11pt]{amsart}
\usepackage{hyperref}
\usepackage{amssymb,amsfonts}
\usepackage{hyperref}
\usepackage{amssymb,textcomp}

\usepackage{enumerate}

\usepackage[utf8]{inputenc}
\usepackage[T1]{fontenc}

\usepackage[mathscr]{eucal}

\usepackage[a4paper, twoside=false, vmargin={3cm,3cm}, includehead]{geometry}

\newtheorem{lemma}{Lemma}[section]
\newtheorem{theorem}[lemma]{Theorem}
\newtheorem*{theorem*}{Theorem}
\newtheorem{corollary}[lemma]{Corollary}

\newtheorem{proposition}[lemma]{Proposition}
\newtheorem*{proposition*}{Proposition}

\newtheorem{problem}{Problem}
\newtheorem*{problem*}{Problem}

\theoremstyle{definition}
\newtheorem*{claim*}{Claim}

\newtheorem*{definition}{Definition}

\newtheorem*{remark}{Remark}
\newtheorem*{remarks}{Remarks}

\newcommand{\wrE}{\mathbb{E}^{w_r}}
\newcommand{\wE}{\mathbb{E}^{w}}
\newcommand{\C}{{\mathbb C}}
\newcommand{\E}{{\mathbb E}}

\newcommand{\N}{{\mathbb N}}

\newcommand{\Q}{{\mathbb Q}}
\newcommand{\R}{{\mathbb R}}
\renewcommand{\S}{\mathbb{S}}
\newcommand{\T}{{\mathbb T}}
\newcommand{\Z}{{\mathbb Z}}
\newcommand{\U}{{\mathbb U}}

\newcommand{\CP}{{\mathcal P}}

\newcommand{\CX}{{\mathcal X}}

\newcommand{\bM}{{\mathbf{M}}}

\newcommand{\bN}{{\mathbf{N}}}

\newcommand{\norm}[1]{\left\Vert #1\right\Vert}

\DeclareMathOperator{\id}{id}

\begin{document}

\title[Furstenberg systems of Hardy field sequences]{Furstenberg systems of Hardy field sequences and applications}
\thanks{The author was supported  by the
	Hellenic Foundation for
	Research and Innovation, 
	Project
	No: 1684.}

\author{Nikos Frantzikinakis}
\address[Nikos Frantzikinakis]{University of Crete, Department of mathematics and applied mathematics, Voutes University Campus, Heraklion 71003, Greece} \email{frantzikinakis@gmail.com}
\begin{abstract}
We study measure preserving systems, called Furstenberg systems,  that model the statistical behavior of   sequences defined by smooth  functions with  at most  polynomial growth. Typical examples
are the sequences  $(n^\frac{3}{2})$, $(n\log{n})$, and
$([n^\frac{3}{2}]\alpha)$, $\alpha\in \mathbb{R}\setminus\mathbb{Q}$, where the entries are taken $\mod{1}$.
 We show that their Furstenberg systems arise from unipotent transformations on  finite dimensional  tori with some invariant measure that is absolutely continuous with respect to the Haar measure and deduce that they are disjoint from every ergodic system.
We also study similar problems for sequences of the form
$(g(S^{[n^{\frac{3}{2}}]} y))$, where $S$ is a measure preserving transformation on the probability space  $(Y,\nu)$, $g\in L^\infty(\nu)$,
 and $y$ is a typical point in $Y$. We prove that the corresponding  Furstenberg systems are strongly stationary and deduce from this a multiple ergodic theorem and a multiple recurrence result for   measure preserving transformations of zero entropy that do not satisfy any commutativity conditions.
\end{abstract}




\maketitle

\centerline{{\em Dedicated to the memory of Michael Boshernitzan}}

\section{Introduction and main results}\label{S:MainResults}

\subsection{Introduction}
A well known  observation of Furstenberg is that   the statistical
behavior of the sequence $(p(n))$ on $\T$ (or the sequence $(e^{2\pi i p(n)})$ on $\S^1$), where $p\in \R[t]$ is an arbitrary polynomial
with real coefficients,
can be modeled by dynamical systems of algebraic nature
(see \cite[Theorem~3.13]{Fu}).
For instance, the statistical behavior of the sequence  $(n^2\alpha)$ on $\T$, where  $\alpha\in \R$, can be modeled by the measure preserving system $(\T^2,m_{\T^2}, S)$ where $m_{\T^2}$ is the Haar measure and $S\colon \T^2\to \T^2$ is defined by
$$
S(x,y):=(x+\alpha,y+x), \qquad x,y\in \T.
$$
To be more precise, for every $f\in C(\T)$, there exists $g\in C(\T^2)$ (in fact, we can choose $g(x,y):=f(y)$, $x,y\in \T$)  such that if $a(n):=f(n^2\alpha)$, $n\in\N$, then
$$
\lim_{N\to\infty}\frac{1}{N} \sum_{n=1}^N \prod_{j=1}^\ell a(n+n_j)=\int \prod_{j=1}^\ell g(S^{n_j}(x,y))\, dm_{\T^2}(x,y)
$$
holds for all $\ell\in \N$ and $n_1,\ldots, n_\ell\in \Z$.

Constructing similar statistical models
 of dynamical nature, which we will later call ``Furstenberg systems'', for other sequences of interest in analytic number theory,
is an intriguing problem that has recently attracted a lot of attention. For instance, it is
conjectured that the Liouville function $\lambda$ can be modeled by a  Bernoulli system (this is equivalent to a conjecture of  Chowla) and  the M\"obius function $\mu$ can be modeled by the direct product of a procyclic system and a Bernoulli system (see \cite{AKLR14,S,Sa12}). At the moment only partial
information about the structure of such measure preserving systems is available; see for example \cite{FH18} for some related progress and
\cite{FH19,GLR20} for some recent results related to more general bounded multiplicative functions.

In this article
we seek to construct dynamical models for sequences arising from smooth non-oscillating functions with at most polynomial growth,
and for convenience we take them to belong to some Hardy field. Typical examples include the sequence
 $(n^{\frac{3}{2}})$ and the sequence $([n^{\frac{3}{2}}]\alpha)$, $\alpha\in\R\setminus\Q$, which are thought of as sequences on $\T$.
We will see  that the first sequence can be modeled by the non-ergodic measure preserving  system  $(\T^2,m_{\T^2}, S)$ where  $S\colon \T^2\to \T^2$ is defined by
$$
S(x,y):=(x,y+x), \qquad x,y\in \T.
$$
The second sequence can be modeled by the direct product of two systems of the previous form.
We  get  similar results when $n^{\frac{3}{2}}$ is replaced by  $n^a$ with $a\in \R_+\setminus \Z$, but with $S$ replaced by a (non-ergodic) unipotent transformation on $\T^d$ where $d:=[a]+1$ (see \eqref{E:unipotent} for the exact form). We also obtain  results  when $n^{\frac{3}{2}}$ is replaced by $n\log n$, or  $n(\log{n})^{\frac{1}{2}}$, or  $n^2\alpha+n^{\frac{3}{2}}$, where $\alpha$ is irrational,  and, perhaps surprisingly,  it turns out that these  four sequences have different dynamical models that are representative for general Hardy field sequences with  at most polynomial growth.
The reader will find  comprehensive results
in Theorems~\ref{T:HardySequence}  and \ref{T:HardySequence'}. 
Using these results and a disjointness argument, we deduce in Corollary~\ref{C:13}  that for all $a\in \R_+\setminus \Z$  and  $\alpha\in \R\setminus \Z$ we have
 	$$
 \lim_{N\to\infty} \frac{1}{N}\sum_{n=1}^N e^{2\pi i n^a} w(n)=0 \qquad \text{ and } \qquad  \lim_{N\to\infty} \frac{1}{N}\sum_{n=1}^N e^{2\pi i [n^a]\alpha} w(n)=0
 $$
 for every ergodic sequence $w\colon \N\to \U$ (a notion defined in Section~\ref{SS:Defs}). Interestingly, in the previous statement  the sequence $n^a$ can be  replaced  by  $n\log n$ but not by $n(\log n)^b$ for any $b<1$, the reason being that all dynamical
  models of the first sequence are disjoint from all ergodic systems but
  some  models of the second sequence are ergodic.

Moreover, we study similar problems for sequences of the form $(g(S^{[n^a]}y))$ where $a\in \R_+\setminus \Z$, $(Y,\nu,S)$ is an arbitrary measure preserving  system, $g\in L^\infty(\nu)$, and $y$ is a typical point in $Y$.
Although it seems hard to determine the exact structure of the  dynamical models of such sequences, we show in Theorem~\ref{T:HardyIterates} that they enjoy a dilation invariance property called ``strong stationarity'' (defined in Section~\ref{SS:sst}), a property that is  not always shared by dynamical models of the above sequences when $a$ is a positive integer.
An important point is that strongly stationary systems
have trivial spectrum and their ergodic components are direct products of infinite-step nilsystems and Bernoulli systems, and  these structural properties imply disjointness from all ergodic zero-entropy systems. This  allows us to deduce in Corollary~\ref{C:HardyIterates}
that if $T,S$ are arbitrary ergodic measure preserving transformations acting on a probability space $(X,\CX,\mu)$ and  the transformation $T$  has zero entropy, then
 for all $a\in \R_+\setminus \Z$ and   $f,g\in L^\infty(\mu)$ we have
$$
	\lim_{N\to\infty} \frac{1}{N}\sum_{n=1}^N T^nf \cdot  S^{[n^a]}g =\int f\, d\mu \int g\, d\mu	
	$$
where the limit is taken in $L^2(\mu)$. We stress that we impose no commutativity assumptions on $T,S$.  Such multiple ergodic theorems are rather rare because the usual toolbox that enables us to study convergence problems for such  averages requires that $T$ and $S$ generate a nilpotent group. Lastly, we note that  for $a\in [1,+\infty)$ the previous averages may diverge
 if we drop the zero-entropy assumption on $T$ (see \cite[Example~7.1]{Be}  and \cite[Section~4]{FLW11}).

\subsection{Definitions and notation} \label{SS:notation}
In order to  facilitate exposition, we introduce some definitions and notation.

For  $N\in\N$  we let  $[N]:=\{1,\dots,N\}$. Let $a\colon \N\to \C$ be a  bounded sequence.   If $A$ is a non-empty finite subset of $\N$  we let
$$
\E_{n\in A}\,a(n):=\frac{1}{|A|}\sum_{n\in A}\, a(n).
$$
We also use a similar notation for finite averages of measures. If $A$ is an infinite subset of $\N$ we let
$$
\E_{n\in A}\, a(n):=\lim_{N\to\infty} \E_{n\in A\cap [N]}\, a(n)
$$
if  the limit exists.

If $(M_k)_{k\in\N}$ is a strictly increasing sequence of positive integers we denote with $\bM$ the sequence of intervals $([M_k])_{k\in\N}$. If $a\colon \N\to \C$ is a bounded sequence we let
$$
\E_{n\in\bM}\, a(n):=\lim_{k\to\infty} \frac{1}{M_k}\sum_{n\in[M_k]}a(n)
$$
if the limit exists.

If $a,b\colon \R_+\to \R$ are functions we write
\begin{itemize}
\item  $a(t)\prec b(t)$ if $\lim_{t\to +\infty} a(t)/b(t)=0$;


\item   $a(t)\sim b(t)$ if $\lim_{t\to +\infty} a(t)/b(t)$ exists and is non-zero;

\item  $a(t)\ll b(t)$ if there exists $C>0$ such that $|a(t)|\leq C|b(t)|$ for all large enough $t\in \R$.
\end{itemize}
In particular, $a(t)\prec 1$ means that $\lim_{t\to +\infty}a(t)=0$. We say that the function  $a\colon \R_+\to \R$ has {\em at most  polynomial growth} if there exists $d\in \N$ such that $a(t)\prec t^d$.

With $\N$ we denote the set of positive integers and with $\Z_+$ the set of non-negative integers.

We often denote sequences on $\N$ or on $\Z$ by  $(a(n))$, instead of $(a(n))_{n\in\N}$ or $(a(n))_{n\in \Z}$; the domain of the sequence is going to be clear from the context.

With  $\R_+$ we denote the set of non-negative real numbers.  For $t\in \R$ we let $e(t):=e^{2\pi i t}$. With   $[t]$ we denote the integer part of $t$ and with $\{t\}$ the fractional part of $t$.

We denote with $\S^1$ the complex unit circle and with $\U$ the complex unit disc. With $\T$ we denote the one dimensional torus and we often identify it with $\R/\Z$ or  with $[0,1)$.
 We often denote elements of $\T$ with  real numbers but we are implicitly  assuming that these real numbers are taken  $\mod{1}$.

\subsection{Results about Hardy field sequences}\label{SS:HardySequence}
We start with a result that describes the possible dynamical systems that
 model the statistical behavior of
Hardy field sequences (see definition in Section~\ref{S:Hardy}) with at most polynomial growth taken $\mod{1}$.  The role of these ``dynamical models'' play the ``Furstenberg systems''
that are associated with these sequences via a variant of a correspondence principle due to Furstenberg; we refer the reader to Section~\ref{SS:Defs} for the definition
and basic facts regarding these  systems.

 It turns out that  the possible Furstenberg systems
   admit an algebraic characterization and have the form  ${\bf X}_d:=(\T^{d+1}, \lambda\times m_{\T^{d}}, S_d)$, where $d$ is the ``degree'' of the sequence,  $\lambda$ is a probability measure on $\T$, and  $S_d$ is the unipotent homomorphism of $\T^{d+1}$  defined by
	\begin{equation}\label{E:unipotent}
S_d(y_0,\ldots, y_d):=(y_0,y_1+y_0,\ldots, y_d+y_{d-1}), \quad y_0,\ldots, y_d\in \T.
\end{equation}
 Note that the measure $\lambda\times m_{\T^{d}}$ is $S_d$-invariant and the system  ${\bf X}_d$ is non-ergodic unless $\lambda$ is a point mass (in which case it  is ergodic if and only if  $\lambda=\delta_\alpha$ for some irrational $\alpha\in \T$).
For a given Hardy field function $a\colon \R_+\to \R$ with at most polynomial growth, the following result determines the structure of all possible Furstenberg systems of
the sequence $(a(n))$ on $\T$ and related sequences.

\begin{theorem}\label{T:HardySequence}
	Let $a\colon \R_+\to \R$ be a Hardy field function with at most polynomial growth and $b\colon\N\to \T$ or $\S^1$ be defined  by $b(n):=a(n) \mod{1}$ 
	 or  $b(n):=e(a(n))$, $n\in \N$.
	\begin{enumerate}
		\item If $t^d\log{t} \prec a(t)\prec t^{d+1}$ for some $d\in \Z_+$, then   $(b(n))$ has a unique Furstenberg system that 
		 is isomorphic to
		the system ${\bf X}_d$   defined above with $\lambda:=m_\T$.

		\item If   $a(t)\sim t^d\log{t}$ for some  $d\in \Z_+$, then   $(b(n))$ does not have  a unique Furstenberg system, and any  Furstenberg system  of $(b(n))$ is  isomorphic to  the system ${\bf X}_d$   defined above for some probability measure   $\lambda \ll m_{\T}$.

		\item If $t^d \prec a(t)\prec t^d\log{t}$ for some $d\in \Z_+$, then   $(b(n))$ does not have  a unique Furstenberg system, and any Furstenberg system  of $(b(n))$  is  isomorphic to  the system ${\bf X}_d$   defined above with $\lambda=\delta_t$ for some $t\in \T$ (and for any such $b$ all measures $\delta_t$, $t\in \T$, arise).
			
		
		\item If  $a(t)=t^d\alpha+\tilde{a}(t)$ for some $d\in \Z_+$ where $\tilde{a}(t)\prec t^d$ and $\alpha$ is irrational, then $(b(n))$ has a unique Furstenberg system  that
		 is  isomorphic to  the system ${\bf X}_d$   defined above where $\lambda=\delta_{\frac{\alpha}{d!}}$,  in particular, it is isomorphic to  a totally ergodic affine transformation on $\T^d$ with the Haar measure.

		
		\item If none of the above applies, then   $a(t)=p(t)+\epsilon(t)+\tilde{a}(t)$ where  $p\in \mathbb{Q}[t]$, $\epsilon(t)\to 0$,  and  $\tilde{a}$ is a Hardy field function that is covered in cases  $(i)$-$(iv)$. In particular,  there exists $r\in \N$ such that for $k=0,\ldots, r-1$ the sequence $b(rn+k)$ is covered in cases  $(i)$-$(iv)$.			 	
	\end{enumerate}
\end{theorem}
\begin{remarks}
	$\bullet$ If $\phi\colon \T\to \C$ is Riemann-integrable, combining the previous result with Proposition~\ref{P:ImageRiemann} below
	we get similar results for the sequence $\phi(a(n))$.
	
	$\bullet$ The systems described in Part~$(i)$ turn out to be strongly stationary (see definition in Section~\ref{SS:sst}). For a related result covering  Hardy field sequences on nilmanifolds see Theorem~\ref{T:NilHardyIterates}
	below.
\end{remarks}



In order to prove the previous result we  show in Lemmas~\ref{L:33} and \ref{L:34} below that the sequence $(e(a(n)))$ has
the same statistical behavior as the sequence $(S_d^nf)$ where $S_d\colon \T^{d+1}\to \T^{d+1}$ is given by \eqref{E:unipotent}  and $f\colon \T^{d+1}\to \C$ is defined by
$f(y):=e(y_d)$ for $y=(y_0,\ldots, y_d)\in \T^{d+1}$.  A key tool that  we use in the proof of this fact is an  equidistribution result of Boshernitzan (see Theorem~\ref{T:Boshernitzan}) that helps us compute the correlations of the first sequence.

A consequence of the previous structural result is the following disjointness statement:
\begin{corollary}\label{C:disjoint}
	Let
	$a\colon \R_+\to \R$ be a  Hardy field function such that   $t^d\log{t}\prec a(t)\prec t^{d+1}$   or $a(t)\sim  t^d \log{t}$ for some  $d\in \Z_+$ and  $b\colon\N\to \T$ or $\S^1$ be defined  by $b(n):=a(n) \mod{1}$ 
	 or  $b(n):=e(a(n))$, $n\in \N$.
	Then all Furstenberg systems of the sequence $b$ are disjoint from all ergodic systems.
\end{corollary}
\begin{remark}
	If $t^d\prec a(t)\prec t^d\log{t}$ for some $d\in \Z_+$, then as shown in Part~$(iii)$ of Theorem~\ref{T:HardySequence} some of the Furstenberg systems of the sequence $(b(n))$ are ergodic.
	\end{remark}
Using the previous result and a disjointness argument we get the following:

\begin{corollary}\label{C:13}
	Let
	$a\colon \R_+\to \R$ be a  Hardy field function such that    $t^d\log{t}\prec a(t)\prec t^{d+1}$ for some $d\in \Z_+$  or $a(t)\sim  t^d \log{t}$ for some  $d\in \N$, and let $b(n):=e(a(n))$ or $b(n):=e([a(n)]\alpha)$, $n\in\N$, where $\alpha \in \R\setminus \Z$. Then
	\begin{equation}\label{E:w}
		\lim_{N\to\infty} \frac{1}{N}\sum_{n=1}^N  b(n) \, w(n) =0
	\end{equation}
	for every ergodic sequence $w\colon \N\to \U$.
\end{corollary}
\begin{remarks}
$\bullet$ Examples of ergodic sequences are all nilsequences, all bounded generalized polynomial sequences
(see \cite{BL07}), or more generally, sequences of the form $(\phi(S^ny))$ where $(Y,\nu,S)$ is a uniquely ergodic system, $y\in  Y$, and $\phi\colon \T\to \C$ is Riemann-integrable with respect to $\nu$. Also several multiple correlation
sequences are known to be ergodic, for example sequences of the form $\int \prod_{j=1}^\ell T_j^{p_j(n)}f_j\, d\mu $, where $T_1,\ldots, T_\ell$ are commuting measure preserving transformations acting on a probability space $(X,\mu)$, $f_1,\ldots, f_\ell\in L^\infty(\mu)$, and $p_1,\ldots, p_\ell\colon \Z\to \Z$ are polynomials (for a proof see \cite[Section~2.2 and Proposition~3.3]{LMR20}).

$\bullet$  If $a(t)\sim \log{t}$, then  our argument gives  for $b(n):= e(a(n))$ or  $b(n):= e([a(n)]\alpha)$,  $n\in\N$, with $\alpha$ irrational,
 that $$
 \lim_{N\to\infty} \big(\E_{n\in [N]} \, b(n) \, w(n)-\E_{n\in [N]} \, b(n) \cdot \E_{n\in [N]}\, w(n)\big)= 0.
 $$

$\bullet$
If $t^d\prec a(t)\prec t^d\log{t}$ for some $d\in \Z_+$, then it can be shown that  \eqref{E:w} fails for some ergodic
sequence $w\colon \N\to\U$. We briefly sketch the argument when    $b(n):=e(a(n))$, $n\in \N$,  and $d=1$.
In this case we have  $a(t):=ta_1(t)$
for some $a_1\colon \R_+\to \R$ with  $1\prec a_1(t)\prec \log{t}$.
  We can choose $M_k\to +\infty$
such that $\{a_1(M_k)\}\to \alpha$. We   let $w(n):=e(-a(n))$ if  $n\in [M_k/2,M_k]$ for some $k\in \N$, and $w(n):=e(-n\alpha)$ otherwise.  Then
 it can be shown that  the sequence $w$ has a unique Furstenberg system  and it is isomorphic 
to  the system $(\T,m_\T,S)$, where $Sx:=x-\alpha$, $x\in \T$
 (the argument is similar to the one used in the proof of Part~ $(iii)$ of Theorem~\ref{T:HardySequence}), hence it is ergodic.    But \eqref{E:w} fails since $\E_{M_k/2\leq n\leq M_k}\, b(n)\, w(n)= 1$  for every $k\in \N$.
\end{remarks}
Another consequence of Theorem~\ref{T:HardySequence} is that under certain growth conditions,
 equidistribution properties of Hardy field sequences remain valid even if one  samples the sequence along  an arbitrary ergodic subsequence.
\begin{corollary}\label{C:13'}
	Let
	$a\colon \R_+\to \R$ be a  Hardy field function such that   $t^d\log{t}\prec a(t)\prec t^{d+1}$ for some $d\in \Z_+$  or $a(t)\sim  t^d \log{t}$ for some  $d\in \N$. Then for every ergodic sequence $b\colon \N\to \N$ and  $\alpha\in \R\setminus\Q$, the  sequences   $\big((a\circ b)(n)\big)$  and $\big([(a\circ b)(n)]\alpha\big)$ are equidistributed  $ \!  \! \mod{1}$,
	and the sequence $\big([(a\circ b)(n)]\big)$ is equidistributed  $\!  \!  \mod{q}$ for every $q\in \N$.
\end{corollary}
\begin{remarks}
	$\bullet$ The case where $b(n)=n$, $n\in \N$, follows from the equidistribution result of  Boshernitzan stated in Theorem~\ref{T:Boshernitzan}. Other examples of ergodic sequences of integers  include the sequences $b(n)=[n\alpha+\beta]$, $n\in \N$, where $\alpha>0$ and $\beta\in \R$. More generally, if $(Y,\nu,S)$ is a uniquely ergodic system,  $U$ is a set of positive measure with boundary of measure zero,  $y_0\in Y$, and $E:=\{n\in\N\colon S^ny_0\in U\}$, then $E$ has positive density and the sequence formed by taking the elements of $E$ in increasing order is an ergodic sequence of integers.
	
$\bullet$ The conclusion fails if $t^d\prec a(t)\prec t^d\log{t}$ for some $d\in \Z_+$, for reasons similar to those described in the third remark after Corollary~\ref{C:13}.
\end{remarks}

Finally, using Theorem~\ref{T:HardySequence} we can also describe the structure of Furstenberg systems
of sequences of the form $([a(n)]\alpha)$ on $\T$, where $a\colon \R_+\to \R$ is a Hardy field function with at most polynomial growth and $\alpha\in \R$. For simplicity we restrict our analysis to the  special case where $t^d\log{t}\prec a(t)\prec t^{d+1}$ for some $d\in \Z_+$ and irrational $\alpha$.
\begin{theorem}\label{T:HardySequence'}
	Let   $a\colon \R_+\to \R$ be a Hardy field function such that   $t^d\log{t}\prec a(t)\prec t^{d+1}$ for some $d\in \Z_+$. Let 	
	$\phi\colon \T\to \C$ be Riemann-integrable,
	and $b(n):=\phi([a(n)]\alpha)$, $n\in\N$,   for some  $\alpha \in \R\setminus \Q$
	and 	$S_d\colon \T^{d+1}\to\T^{d+1}$  be given by \eqref{E:unipotent}. Then
	$b(n)$ has a unique Furstenberg system and it  
	is a factor of the system $(\T^{2(d+1)},m_{\T^{2(d+1)}},S_d\times S_d)$.
\end{theorem}


\subsection{Results about Hardy field iterates}
Let $(Y,\nu,S)$ be a measure preserving system and $g\in L^\infty(\nu)$. The next result gives structural information
 on the Furstenberg systems of sequences of the form $(g(S^{[a(n)]}y))$ for typical values of $y\in Y$ (we refer the reader to Sections~\ref{SS:mps} and \ref{SS:Defs}
 for explanations regarding the terminology used).
\begin{theorem}\label{T:HardyIterates}
	Let $a\colon \R_+\to \R$ be a  Hardy field function such that $t^{d+\varepsilon}\prec a(t)\prec t^{d+1}$ for some $d\in \Z_+$ and $\varepsilon>0$. Furthermore, let  $(Y,\nu,S)$ be a measure preserving system.
	Then every strictly increasing sequence of positive integers $(N_k)$ has a subsequence $(N_k')$ such that for almost every $y\in Y$ and for
	every $g\in L^\infty(\nu)$  the sequence $(g(S^{[a(n)]}y))$ admits correlations  on $\bN':=([N'_k])_{k\in\N}$ and the corresponding Furstenberg system  has trivial spectrum,\footnote{We say that a system $(X,\mu,T)$ has trivial spectrum if $Tf=e^{2\pi i \alpha}f$ for some $\alpha \in [0,1)$ and  non-zero $f\in L^2(\mu)$,  implies that $\alpha=0$.}   and its ergodic components are isomorphic to direct products of infinite-step nilsystems and Bernoulli systems.
\end{theorem}
\begin{remarks}
$\bullet$ 	It is expected  that for almost every $y\in Y$ for every $g\in L^\infty(\nu)$ the sequence $(g(S^{[a(n)]}y))$
	has a unique Furstenberg system; but  this is equivalent to a  pointwise convergence result for multiple ergodic averages that at the moment seems out of reach.

	$\bullet$ If $(Y,\nu,S)$ is a weak mixing system and $d\in \N$,
then using Theorem~\ref{T:BMR} below (or \cite[Theorem~A]{BH09}) it is not hard to show that  in the conclusion of Theorem~\ref{T:HardyIterates} all the Furstenberg  systems can be taken to be Bernoulli systems. On the other hand, if
$(Y,\nu,S)$ is a non-trivial infinite-step nilsystem, then it is possible to show that the corresponding Furstenberg systems are non-ergodic
and their  ergodic components are infinite-step nilsystems.

$\bullet$ At the expense of using  a different averaging scheme  we can relax the growth assumption on $a(t)$ to  the assumption	 $t^{d}\prec a(t)\prec t^{d+1}$ for some $d\in \Z_+$. To prove  this, two  modifications are needed in our argument:  The first is in the definition of Furstenberg systems, one has to use the weighted averages $\wE_{n\in\N}$  for $w:=a^{(d)}$ (see Section~\ref{SS:sstnil} for their definition)
in place of the usual Ces\`aro averages. The second is in Section~\ref{SS:sstnil}, one has to use the corresponding equidistribution results from \cite{BMR20} for the weighted averages.  The same comment applies for Corollary~\ref{C:HardyIterates}.	

$\bullet$ When   $a(n)=n^2$, $n\in\N$,  the corresponding Furstenberg systems may have non-trivial spectrum. For example let $S\colon \T\to\T$ be given by $Sx:=x+\alpha$, $x\in \T$, for some irrational $\alpha$,  and $g(y):=e^{2\pi i y}$, $y\in \T$. Then it is not hard  to show that for every $y\in \T$ the sequence $(g(S^{n^2}y))$ has a unique Furstenberg system and it is isomorphic to the system
$(\T^2,m_{\T^2},R)$ where  $R$ of $\T^2$ is defined by  $R(z,w):=(z+\alpha,w+z)$, $z,w\in \T$. We also remark
  that this system  is not strongly stationary; this  is in contrast with Theorem~\ref{T:SstHardyIterates} below (which covers the case of fractional powers).	
\end{remarks}
The proof of Theorem~\ref{T:HardyIterates} is less direct than the one of Theorem~\ref{T:HardySequence} because it appears to be hard to exhibit precise
systems that model the statistical behavior of the sequence $(g(S^{[a(n)]}y))$.
Instead, we proceed by showing in Theorem~\ref{T:SstHardyIterates}  that the Furstenberg systems
of such sequences are strongly stationary (a property that fails when the sequence  $(a(n))$ is polynomial). The proof of this fact follows from the multiple ergodic theorem of Proposition~\ref{P:sstformula}, which in turn is proved using  recent deep
 results of Bergelson, Moreira, and Richter~\cite{BMR20}, using the theory of characteristic factors of Host-Kra~\cite{HK05} and equidistribution results on nilmanifolds. The structure of strongly stationary systems was determined in \cite{FLW11,Jen97} and    we use these structural results as a black box   in order to complete the proof of Theorem~\ref{T:HardyIterates}.

Using the structural result of Theorem~\ref{T:HardyIterates} and a disjointness argument we deduce the following multiple ergodic theorem and a corresponding  multiple recurrence result:
\begin{corollary}\label{C:HardyIterates}
	Let $a\colon \R_+\to \R$ be a  Hardy field function such that $t^{d+\varepsilon}\prec a(t)\prec t^{d+1}$ for some $d\in \Z_+$ and $\varepsilon>0$. Furthermore,  let $(X,\CX,\mu)$ be a probability space and $T,S\colon X\to X$ be measure preserving transformations (not necessarily commuting). Suppose that the system $(X,\mu,T)$ has zero entropy. Then
	\begin{enumerate}
\item 		For every    $f,g\in L^\infty(\mu)$ we have
	\begin{equation}\label{E:formula'}
	\lim_{N\to\infty} \frac{1}{N}\sum_{n=1}^N
	T^nf \cdot  S^{[a(n)]}g =\E(f|\mathcal{I}_T) \cdot  \E(g|\mathcal{I}_S)
	\end{equation}
	where the limit is taken in $L^2(\mu)$.
\item For every $A\in \CX$ we have
	\begin{equation}\label{E:formula''}
\lim_{N\to\infty}\frac{1}{N}\sum_{n=1}^N
\mu(A\cap T^nA\cap S^{[a(n)]}A)\geq (\mu(A))^3.
	\end{equation}
\end{enumerate}
\end{corollary}
\begin{remarks}
$\bullet$	Using weighted averages and the second remark after Theorem~\ref{T:HardyIterates} we can get a variant of \eqref{E:formula'} and use it to deduce   that
	if a Hardy field function $a\colon \R_+\to \R$   satisfies $t^d\prec a(t)\prec t^{d+1}$ for some $d\in \Z_+$, then  for every $A\in \CX$ and $\varepsilon>0$ we have
	$$
	\limsup_{N\to\infty}\frac{1}{N}\sum_{n=1}^N \mu(A\cap T^nA\cap S^{[a(n)]}A)\geq (\mu(A))^3-\varepsilon.
	$$

$\bullet$ The conclusion of Corollary~\ref{C:HardyIterates} fails if we do not assume that $(X,\mu,T)$ has  zero entropy, since for every strictly increasing sequence of integers $(b(n))$ and every $c\colon \N\to [-1,1]$
there exist Bernoulli systems $(X,\mu,T)$ and  $(X,\mu,S)$, and $f,g\in L^\infty(\mu)$, such that
 $c(n):=\int  T^nf\cdot  S^{b(n)}g\, d\mu$, $n\in\N$,   (see \cite[Section~4]{FLW11}), and as a consequence  the
 averages $\frac{1}{N}\sum_{n=1}^N\,\int  T^nf\cdot  S^{b(n)}g\, d\mu $
 do not always converge.

$\bullet$ If we assume that the transformations $T,S$ commute,  then it is known by  \cite{DerL96} that the Pinsker factor is characteristic for pointwise convergence of the averages in \eqref{E:formula'} and as a consequence for mean  convergence. Hence, in this case, we get mean convergence in  \eqref{E:formula'}, \eqref{E:formula''}  without the assumption that the system $(X,\mu,T)$ has zero entropy (but in the commutative case this  can also be obtained by using the method of \cite{F15}).
	\end{remarks}

\subsection{Open problems}
For a given ergodic system $(Y,\nu, S)$ and function  $g\in L^\infty(\nu)$,
it is also natural to study the possible Furstenberg systems of sequences of the form  $(g(S^{p(n)}y))$ where $p$ is a polynomial with integer coefficients and $y$ is a typical point in $Y$.
When $p(n)=n$,  it is an easy consequence of the pointwise ergodic theorem that    for almost every $y\in Y$ the sequence $(g(S^{n}y))$  has a unique Furstenberg system and it is a factor of the system $(Y,\nu, S)$ (see Proposition~\ref{P:ergodic}).
  The situation is dramatically different when one
  considers non-linear polynomials  in which case one expects sever restrictions
  on the structure of the possible Furstenberg systems. Furthermore, different Furstenberg systems arise than those arising in Theorem~\ref{T:HardyIterates}.
\begin{problem}
 Let $(Y,\nu, S)$ be a system,  $p\in \Z[t]$ be a non-linear polynomial, and $g\in L^\infty(\nu)$. Show that
 for almost every $y\in Y$ the sequence $(g(S^{p(n)}y))$ has a unique Furstenberg system that is
 ergodic and isomorphic to a direct product of an infinite-step nilsystem and a Bernoulli system.
\end{problem}
If one assumes in Problem 1 uniqueness of the Furstenberg system, then using the multiple ergodic theorem
from \cite{B87} it is easy to deduce that for $(Y,\nu, S)$  weak mixing,  and $p\in \Z[t]$ non-linear,  for almost every $y\in Y$ the Furstenberg system of the  sequence $(g(S^{p(n)}y))$ is a Bernoulli system.
On the other hand, proving uniqueness  of the Furstenberg system seems very hard as this amounts to proving a pointwise convergence result
for multiple ergodic  averages that currently seems out of reach. So as a first step for an unconditional result, one probably has to compromise with
a result in the spirit of Theorem~\ref{T:HardyIterates}  that describes some of the possible Furstenberg systems  of the sequence $(g(S^{p(n)}y))$.

Lastly, it would be interesting to know if a variant of Corollary~\ref{C:HardyIterates} holds when $a(t)$ is a polynomial; this is the context of the next problem.

\begin{problem}
 Let  $(X,\CX,\mu)$ be a probability space, and  $T,S\colon X\to X$ be measure preserving transformations. Suppose that the system $(X,\mu,T)$ has zero entropy  and $f,g\in L^\infty(\mu)$.

 \begin{enumerate}
 \item Is it true that
 the averages
 $$
\lim_{N\to\infty} \frac{1}{N}\sum_{n=1}^N
T^nf\cdot  S^{p(n)}g
$$
converge in $L^2(\mu)$ when $p(n)=n$ or $p(n)=n^2$?

\item Is it true that for every $A\in \CX$ with $\mu(A)>0$ there exists $n\in\N$ such that
$$
\mu(A\cap T^nA\cap S^{p(n)}A)>0
$$
 when $p(n)=n$ or $p(n)=n^2$?
 \end{enumerate}
\end{problem}
The method used to prove Corollary~\ref{C:HardyIterates} does not give a positive result in this case. The reason is that for typical $x\in X$ for every $f\in L^\infty(\mu)$ the Furstenberg systems of the sequences
$(f(T^nx))$ and $(g(S^{p(n)}x))$ are not always disjoint.
We also remark that  Questions $(i)$ and $(ii)$ have a negative answer if one drops the zero entropy assumption
(see \cite[Example~7.1]{Be} and  \cite[Page~40]{Fu}, or 	\cite{BL04},
  for $p(n)=n$, and \cite[Section~4]{FLW11} for general polynomial $p$).

One can also ask similar questions for averages of the form
 $$
\frac{1}{N}\sum_{n=1}^N
T^{[n^a]}f\cdot  S^{[n^b]}g
$$
where $a,b>1$ are distinct non-integers. We remark that if either $a$ or $b$ is in $(0,1)$, then a relatively  simple argument gives  mean convergence without any assumption on $T$ and $S$ (for $a,b\in (0,1)$ see  \cite[Proposition~6.4]{Fr10}  or
\cite{DKS20}).


\section{Background in ergodic theory}

\subsection{Measure preserving systems}\label{SS:mps}

Throughout the article, we make the standard assumption that all probability  spaces $(X,\CX,\mu)$ considered are Lebesgue, meaning, $X$  can be given the structure of a compact metric space
and $\CX$ is  its Borel $\sigma$-algebra.
A {\em measure preserving system}, or simply {\em a system}, is a quadruple $(X,\CX,\mu,T)$
where $(X,\CX,\mu)$ is a probability space and $T\colon X\to X$ is an invertible, measurable,  measure preserving transformation. We typically omit the $\sigma$-algebra $\CX$  and write $(X,\mu,T)$. The system is {\em ergodic} if the only sets that
are  left invariant by $T$ have measure $0$ or $1$. It  is {\em totally ergodic} if the system $(X,\mu,T^n)$ is ergodic for every $n\in \N$. It is {\em weak mixing} if the system $(X\times X, \mu\times \mu, T\times T)$ is ergodic. We say that $\lambda\in \mathbb{S}^1$ is an {\em eigenvalue} of $(X,\mu,T)$ if there exists non-zero $f\in L^2(\mu)$ such that $Tf=\lambda f$.
Throughout,  for $n\in \N$ we denote  by $T^n$   the composition $T\circ  \cdots \circ T$ ($n$ times) and let $T^{-n}:=(T^n)^{-1}$ and $T^0:=\id_X$. Also, for $f\in L^1(\mu)$ and $n\in\Z$ we denote by  $T^nf$ the function $f\circ T^n$.

In order to avoid unnecessary repetition,  we refer the reader to \cite{FH18, HK18}  for some other standard notions from ergodic theory.
In particular, the reader will find in Section~2 and
in  Appendix~A of \cite{FH18} the definition
of the terms factor,  conditional expectation with respect to a factor, isomorphism, inverse limit, infinite-step nilsystem, infinite-step nilfactor, Bernoulli system,    ergodic decomposition, joining, and disjoint systems; all these notions  are used in this article.

\subsection{Strong stationarity}\label{SS:sst}
We define here a notion that plays a crucial role in the proof of Theorem~\ref{T:HardyIterates} and Corollary~\ref{C:HardyIterates}.
\begin{definition}
	Let $(X,\mu,T)$ be a system. We say that
	\begin{itemize}
		\item a conjugation closed sub-algebra   $\mathcal{F}$ of $L^\infty(\mu)$ is {\em $T$-generating}, if the linear span of  elements of the form $T^nf$ with $f\in \mathcal{F}$ and $n\in \N$, is dense in $L^2(\mu)$.
		
		\item the system $(X,\mu,T)$  is {\em strongly stationary}, if there exists a $T$-generating set  $\mathcal{F}$  such that for every $r\in \N$ we have
		\begin{equation}\label{E:sst}
		\int \prod_{j=1}^\ell T^{n_j}f_j\, d\mu=\int \prod_{j=1}^\ell T^{rn_j}f_j\, d\mu
		\end{equation}
		for all $\ell\in \N$,  $n_1,\ldots, n_\ell\in \Z$, and
		$f_1,\ldots, f_\ell\in \mathcal{F}$.
	\end{itemize}
\end{definition}
\begin{remark}
	It follows from \cite{Jen97} that a system  $(X,\mu,T)$ is strongly stationary if and only if  there exists a $T$-generating set  $\mathcal{F}$ and measure preserving maps $\tau_n$ on $(X,\mu)$, $n\in \N$, such that
	$f(\tau_n x)=f(x)$, $f\in \mathcal{F}$, and $(T\tau_n)(x)=(\tau_nT^n)(x)$ for every  $x\in X$ and $n\in \N$.
\end{remark}

It is easy to verify that Bernoulli systems are strongly stationary.
It is shown in \cite{Jen97} that if an ergodic system is  strongly stationary, then it is necessarily Bernoulli. An example of a non-ergodic strongly stationary system is given by the transformation
$T\colon \T^2\to \T^2$ with the Haar measure $m_{\T^2}$, defined  by
$$
T(x,y):=(x,y+x), \qquad x,y\in \T.
$$
The reader can verify that the set $\mathcal{F}:=\{f(y)\colon f\in L^\infty(m_\T)\}$ is $T$-generating and \eqref{E:sst} is satisfied. In a similar fashion, it can be shown that the
systems ${\bf X}_d$, defined by the transformation $S_d$ in \eqref{E:unipotent},  are strongly stationary when we take the Haar measure $m_{\T^{d+1}}$ on $\T^{d+1}$.

The structure of general strongly stationary systems was determined in \cite{Fr04}. We will use the following structural consequence of
the main results in \cite{Fr04, Jen97}:
\begin{theorem}\label{T:Sst}
	Strongly stationary systems have trivial spectrum and their ergodic components are direct products
	of infinite-step nilsystems and Bernoulli systems.
\end{theorem}

\subsection{Furstenberg systems of sequences}\label{SS:Defs}

In this subsection we reproduce  the notion of a Furstenberg system
from \cite{FH18} in a slightly more general context and record some basic related facts that will be used later.
\begin{definition}
	Let $(Y,d)$ be a compact metric space and  $ \bM:=([M_k])_{k\in\N}$ be a sequence of intervals with $M_k\to \infty$.
	We say that a finite collection of bounded sequences $a_1,\ldots, a_\ell\colon \Z\to Y$
	{\em admits joint  correlations  on $\bM$}, if the   limits
	\begin{equation}\label{E:CorDef}
	\lim_{k\to\infty}\E_{m\in [M_k]} \prod_{j=1}^s f_j(\tilde{a}_j(m+n_j))
	\end{equation}
	exist for all $s \in \N$, all   $n_1,\ldots, n_s\in \Z$
	(not necessarily distinct), all $f_1,\ldots, f_s\in C(Y)$,
	and all $\tilde{a}_1,\ldots,\tilde{a}_s\in \{a_1,\ldots,a_\ell\}$.
\end{definition}
\begin{remarks}
	$\bullet$	 Given $a_1,\ldots, a_\ell\colon \Z \to Y$, since $C(Y)$ is separable,  using a diagonal argument, we get that every sequence of intervals $\bM=([M_k])_{k\in \N}$
	has a subsequence $\bM'=([M_k'])_{k\in\N}$, such that the sequences  $a_1,\ldots, a_\ell$  admit joint correlations on $\bM'$.
	
	$\bullet$ Let $X:=(Y^\ell)^\Z$. Note that  the algebra generated by functions of the form $x\mapsto h(x_j(k))$, $x\in X$, for  $j\in \{1,\ldots, \ell\}$, $k\in  \Z$, and $h\in C(Y)$, separates points in $C(X)$, where
 $x:=(x(n))_{n\in\Z}=(x_1(n),\ldots, x_\ell(n))_{n\in \Z}$. We conclude that  if the
	sequences $a_1,\ldots, a_\ell\colon \Z\to Y$
	 admit joint  correlations  on $\bM$, then  for all $f\in C(X)$ the following limit exist
	$$
	\lim_{k\to\infty}\E_{m\in [M_k]}\,  f(T^ma),
	$$
	where $a:=(a_1,\ldots, a_\ell)\in X$ and $T$ is the shift transformation on $X$, which is defined by $(Tx)(n):=x(n+1)$, $n\in\Z$, $x\in X$.
	Hence, the  weak-star limit $\lim_{k\to\infty}\E_{m\in [M_k]}\, \delta_{T^ma}$
	exists.	
\end{remarks}
If a finite collection of sequences  admits joint correlations on a given sequence of intervals,  then we use a  variant of the correspondence principle of Furstenberg~\cite{Fu77, Fu} in order
to associate a measure preserving system that captures the statistical properties of these sequences.
\begin{definition}\label{D:correspondence}
	Let $(Y,d)$ be a compact metric space  $\ell\in\N$ and $a_1,\ldots, a_\ell\colon \Z\to Y$ be sequences that  admit joint
	correlations  on
	$\bM:=([M_k])_{k\in\N}$. We let  $\mathcal{A}:=\{a_1,\ldots,a_\ell\}$, $X:=(Y^\ell)^\Z$,  $T$ be the shift transformation on $X$,  defined by $(Tx)(n):=x(n+1)$, $n\in\Z$, $x\in X$,
	and $\mu$ be the weak-star limit $\lim_{k\to\infty}\E_{m\in [M_k]}\delta_{T^ma}$
	where $a:=(a_1,\ldots, a_\ell)$ is thought of as an element of $X$.
	\begin{itemize}
	\item We call $(X,\mu,T)$
	the {\em joint Furstenberg system
		associated with} $\mathcal{A}$ on  $\bM$, or  simply, the {\em F-system of $\mathcal{A}$ on $\bM$}.
	
	\item  We say that the finite collection
	$\mathcal{A}$ has a {\em unique } Furstenberg system,  if the weak-star limit $\lim_{M\to\infty}\E_{m\in [M]}\delta_{T^ma}$ exists, or equivalently,  if  $\mathcal{A}$ admits joint correlations on $([M])_{M\in\N}$.
	
	\end{itemize}
\end{definition}
\begin{remarks}
	$\bullet$ If we are given sequences $a_1,\ldots, a_\ell\colon \N\to Y$, we extend them to $\Z$ in an arbitrary way; then  the measure $\mu$ will  not depend on the extension.

$\bullet$ A sequence may  not admit correlations
 on $([M])_{M\in\N}$, so with our definition it may not have a unique Furstenberg system, but nevertheless all its Furstenberg systems could be measure theoretically isomorphic. This happens for example when $a(n):=\{\log\log{n}\}$, $n\in\N$; in this case all Furstenberg systems are isomorphic to the trivial one point system, but $(a(n))$ does not  admit correlations
 on $([M])_{M\in\N}$.

$\bullet$ It follows from \cite[Proposition~3.8]{DGS76} that if a collection of  sequences does not have a unique Furstenberg system on $([M])_{M\in\N}$, then it has uncountably many Furstenberg systems.

$\bullet$ A collection of sequences $a_1,\ldots, a_\ell\colon \Z\to \mathbb{U}$ may have several non-isomorphic  Furstenberg systems
depending on which sequence of intervals $\bM$ we use in the evaluation of  their joint correlations. We call any such system {\em a (joint) Furstenberg system of $a_1,\ldots, a_\ell$}.
\end{remarks}

	If $Y=\S^1$,  $\ell=1$,   $a_1=a$,
and  $F_0\in C(X)$ is defined by  $F_0(x):=x(0)$, $x \in X$,  then letting $z^1:=z$, $z^{-1}:=\overline{z}$ for $z\in \C$, we get that the following identities hold
(and in fact characterize the measure $\mu$)
\begin{equation}\label{E:correspondence}
\E_{m\in {\bM} }\,  \prod_{j=1}^s a^{\epsilon_j}(m+n_j) =\int  \prod_{j=1}^s
T^{n_j}F_0^{\epsilon_j} \, d\mu
\end{equation}
for all  $s\in \N$, $n_1, \ldots, n_s\in \Z$,  $\epsilon_1,\ldots, \epsilon_s\in\{-1,1\}$.
Moreover, if all the limits on the left hand side of \eqref{E:correspondence} exist, then the sequence $(a(n))$ admits  correlations on $\bM$.

In practice, in order to describe the structure of the Furstenberg system of a sequence $a\colon \N\to \U$ on $\bM$, we try to find  a
closed formula for the correlations on the left hand side of \eqref{E:correspondence} (see for example Lemma~\ref{L:33}) and then try to figure out a simple system and a function that has the same correlations (see for example Lemma~\ref{L:34}). If this is not feasible, then we try to obtain some partial information about these correlations that gives us useful feedback for the structure of the Furstenberg systems (see for example Theorem~\ref{T:SstHardyIterates}, which is based on Proposition~\ref{P:sstformula}).

Using the previous definition we can associate ergodic properties to arbitrary bounded sequences of complex numbers and also to strictly increasing sequences of integers  with range a set of positive density.
\begin{definition}	With $\U$ we denote the complex unit disc.  We say that:
	\begin{itemize}
		\item A sequence $a\colon \N\to  \U$ is {\em  ergodic} (or has {\em zero entropy}), if all its Furstenberg systems have the corresponding property.
		
		\item 	A  sequence $a\colon\N\to\N$ is  {\em  ergodic}  (or has {\em zero entropy}),  if it is strictly increasing,
		its range $E:=a(\N)$ is a set of positive density, and
		the $\{0,1\}$-valued sequence ${\bf 1}_E$ is  ergodic (respectively, has zero entropy).
	\end{itemize}
\end{definition}




The following lemma is a simple consequence of the definitions and will be used in order to establish strong stationarity for certain bounded sequences.
\begin{lemma}\label{L:sst}
	\begin{enumerate}
\item Suppose that the sequence $a\colon \N\to \U$ admits correlations on  $\bM$ and for every $s\in \N$, and $n_1,\ldots, n_s\in \Z$, the correlations
$$
\E_{m\in\bM}\prod_{j=1}^s a_j(m+rn_j)
$$
are independent of $r\in \N$, for all $a_1,\ldots, a_s\in \{a,\overline{a}\}$. Then the Furstenberg system of the sequence  $(a(n))$ on $\bM$ is strongly stationary.

\item Let $(Y,d)$ be a compact metric space and suppose that the sequence
$a\colon \N\to Y$ admits correlations on $\bM$ and for every  $s \in \N$ and    $n_1,\ldots, n_s\in \Z$, the correlations
$$
\E_{m\in \bM} \prod_{j=1}^s f_j(a(m+rn_j))
$$
are independent of $r\in \N$ for all  $f_1,\ldots, f_s\in C(Y)$. Then the Furstenberg system of the sequence  $(a(n))$ on $\bM$ is strongly stationary.
\end{enumerate}
\end{lemma}

\subsection{Furstenberg systems of images of sequences}
For a given sequence $a\colon \N\to Y$ and ``regular'' function $\phi\colon Y\to \C$ we would like  to relate Furstenberg systems of sequences of the form $(\phi(a(n)))$ to those of the sequence $(a(n))$.
\begin{definition} Let $\nu$ be a Borel measure on the compact metric space $(Y,d)$.	
	\begin{itemize}
\item We say that $\phi \colon Y\to \R$ is {\em Riemann-integrable with respect to $\nu$}, if  it is Borel measurable and for every $\varepsilon>0$
there exist $\phi^-,\phi^+\in C(Y)$ such that $\phi^-(y)\leq \phi(y)\leq \phi^+(y)$ for every $y\in Y$ and
$$
\phi^-(y)\leq \phi(y)\leq \phi^+(y),\,  y\in Y,\,  \text{ and }\,
\int (\phi^+-\phi^-)\, d\nu\leq \varepsilon.$$

\item We say that a complex valued function $\phi\colon Y\to \C$ is {\em Riemann-integrable with respect to $\nu$} if its real and imaginary parts are Riemann-integrable.

\end{itemize}
\end{definition}
\begin{remark}
	It can be shown that  $\phi\colon Y\to \C$ is Riemann-integrable if the set of discontinuity points of $\phi$ has $\nu$-measure $0$.
	\end{remark}
The next result gives information about the possible Furstenberg systems of  images of sequences under Riemann-integrable functions. Its proof
is based on  some   rather standard approximation arguments; for readers convenience we include some details.
\begin{proposition}\label{P:ImageRiemann}
	Let
	$(Y,d)$ be a compact metric space.  Suppose that the sequences    $a_1,\ldots, a_\ell\colon \Z\to Y$    admit joint 	correlations  on
	$\bM:=([M_k])_{k\in\N}$ and let $(X,\mu,T)$ be their joint  Furstenberg system on $\bM$.
	For  $a:=(a_1,\ldots, a_\ell)$ let
	\begin{equation}\label{E:nu}
	\nu:=\lim_{k\to\infty}\E_{m\in [M_k]}\, \delta_{a(m)},
	\end{equation}
	where the limit is taken in the weak-star sense, and  suppose that the function
	$\phi\colon Y^\ell\to \C$ is Riemann-integrable  with respect to the measure $\nu$.
	Then
	  the sequence
	  $$
	  b(n):=\phi(a_1(n), \ldots, a_\ell(n)), \qquad n\in \N,
	  $$
	   admits correlations on  $\bM:=([M_k])_{k\in\N}$, the corresponding Furstenberg system is a factor of the system  $(X,\mu,T)$, and
	 if $\phi$ is injective,  it  is isomorphic to the system  $(X,\mu,T)$.
\end{proposition}
\begin{proof}	
	We first remark that the existence of the weak-star limit in \eqref{E:nu} follows from our assumption that the sequences $a_1,\ldots, a_\ell$ admit joint correlations on $\bM$.
	
	Let $(X', \mu',T')$ be the Furstenberg system of $b$ (recall that $X'=\U^\Z$).
	We first show that $b$ admits correlations on $\bM$, or equivalently,
	that 	\begin{equation}\label{E:exists'}
	\lim_{k\to\infty}\E_{m\in [M_k]}
	\prod_{j=1}^s f_j(\phi(a_1(n+n_j), \ldots, a_\ell(n+n_j)))
	\end{equation}
	exists for all $s \in \N$, all   $n_1,\ldots, n_s\in \Z$
	(not necessarily distinct), and all $f_1,\ldots, f_s\in C(\U)$.
	By density with respect to the uniform norm we can
	assume that the functions $f_1,\ldots, f_s$ are Lip-continuous on $\U$. Using this and by  approximating $\phi$ in $L^1(\nu)$ we get that in order to show that the averages in \eqref{E:exists'} form a Cauchy sequence for every $\phi$ that is Riemann-integrable with respect to the measure $\nu$, it suffices to show that they form a Cauchy sequence for every $\phi\in C(Y^\ell)$. But if
	$\phi\in C(Y^\ell)$, then the averages  \eqref{E:exists'} converge
	 since by  our assumption  the sequences $a_1,\ldots, a_\ell$ admit joint correlations on $\bM$.

Next, 	we define the  map $\Phi\colon X\to X'$  by
	$$
	\Phi((x_1(n),\ldots, x_\ell(n))_{n\in\Z}):=(\phi(x_1(n),\ldots, x_\ell(n)))_{n\in\Z}.
	$$
	 Since $\phi$ is Borel measurable, the map $\Phi$ is a measurable map and we clearly have that $T'\circ \Phi=\Phi\circ T$.
	Note also that if $\phi$ is one to one, then so is $\Phi$.
	It remains to  show that  $\mu'=\mu\circ \Phi^{-1}$. To this end, let  $f\in C(X')$.
	
	If $
	\phi$ is continuous, then $f\circ\Phi\in C(X)$, hence
	\begin{multline*}
	\int f \, d(\mu\circ\Phi^{-1})= \int f\circ \Phi\, d\mu=\E_{m\in \bM} \, (f\circ \Phi)(T^ma)=\\
	=\E_{m\in \bM}\,  f(T'^m(\Phi\circ a))=\E_{m\in \bM} \, f(T'^mb)=\int f\, d\mu'.
	\end{multline*}
	Hence,  $\mu'=\mu\circ \Phi^{-1}$.
	
	To get a similar identity when $\phi$ is Riemann-integrable   with respect to $\nu$, the only part that needs justification is  that the identity
	\begin{equation}\label{E:Phi}
	\int f\circ \Phi\, d\mu=\E_{m\in \bM} \, (f\circ \Phi)(T^ma)
	\end{equation}
	holds for every $f\in C(X')$.
	Using uniform approximation and linearity we can assume that $f$ is a cylinder function, meaning, of the form $f(x')=\prod_{j=1}^s F_{n_j}(x')$,  for some $s\in \N$ and $n_1,\ldots,n_s\in \Z$,  where for $i\in \Z$ we let $F_i(x'):=x'(i)$, $x'\in X'$. Writing elements $x\in X$ as $x=(x_1,\ldots, x_\ell)$, where $x_1,\ldots, x_\ell\in X',$ we have
	$$
	f(\Phi(x))=\prod_{j=1}^s F_{n_j}(\Phi(x))=\prod_{j=1}^s \phi(x_1(n_j),\ldots, x_\ell(n_j)).
	$$
	Hence,  in order to verify that  \eqref{E:Phi} holds
	 it suffices  to show that
	\begin{equation}\label{E:needed}
	\int \prod_{j=1}^s \phi_j(x_1(n_j),\ldots, x_\ell(n_j))\, d\mu(x)=\E_{m\in \bM} \prod_{j=1}^s \phi_j(a_1(m+n_j),\ldots, a_\ell(m+n_j))
	\end{equation}
	whenever $\phi_1, \ldots, \phi_s\colon Y^\ell\to \C$
	are Riemann-integrable with respect to $\nu$ (it is convenient to
	prove this more general version with $s$ different functions).
	Furthermore, for $j=1,\ldots, s$, writing  $\phi_j$ as a linear combination (over $\C$) of four non-negative real valued  functions that are  Riemann-integrable with respect to $\nu$,  and  using linearity, we see that it suffices to  verify the previous identity when the functions $\phi_1,\ldots, \phi_s$ are real valued and take values in $[0,1]$.
	Since \eqref{E:needed} holds for continuous functions  $\phi_1,\ldots, \phi_s$,
	using   a standard approximation argument
	from above and below by continuous functions and \eqref{E:nu}, we get that \eqref{E:needed} holds for Riemann-integrable functions with respect to $\nu$ as well.  This completes the proof. 	
\end{proof}

We will use the previous result in the proof of Theorem~\ref{T:HardySequence'} in order to show that under suitable assumptions on the sequence $a\colon \N\to \R$, all
Furstenberg systems of the sequence $(e([a(n)]\alpha))$
are factors of joint Furstenberg systems of the sequences $(a(n))$ and $(a(n)\alpha)$ (thought of as sequences on $\T$). These three  sequences are linked via the identity $e([a(n)]\alpha)=\phi(a(n)\alpha, a(n))$, $n\in\N$,  where $\phi\colon \T^2\to \S^1$ is defined by
$\phi(x,y)=e(x -\{y\}\alpha)$, which is Riemann-integrable with respect to $m_{\T^2}$.

One  way to establish uniqueness and determine the structure of the Furstenberg system of a sequence, is to represent it as the image under a ``regular'' function of an orbit of a point in a uniquely ergodic system. This is the context of the next result (which is well known to experts).
\begin{corollary}
	Let  $(Y,\nu,S)$  be a uniquely ergodic system  and $g\colon Y\to \C$ be Riemann-integrable with respect to $\nu$.
		Then for every $y\in Y$ the sequence $(g(S^ny))$ has a unique Furstenberg system that  is a factor of the system
		$(Y,\nu,S)$. Furthermore, if  $g$ is injective, then  we have an isomorphism.
\end{corollary}
\begin{proof}
 	Let $y\in Y$.
By unique ergodicity we get that the sequence
$(S^ny)$ has a unique Furstenberg system that is isomorphic to the system $(Y,\nu,S)$. Moreover, the
 weak-star limit defined in equation  \eqref{E:nu} of  Proposition~\ref{P:ImageRiemann} is equal to $\nu$.
 The result now follows from Proposition~\ref{P:ImageRiemann}.
\end{proof}
It is easy to deduce from the previous result that the sequence $(\sin{n})$ has a unique Furstenberg system that is an ergodic rotation on the circle. Moreover,
for $c_1,c_2, \alpha,\beta\in \R$, the sequences $(c_1\cos(n\alpha)+c_2\sin(n\beta))$ and  $(c_1{\bf 1}_{[0,1/2]}(\{n\alpha\})+c_2 {\bf 1}_{[1/2,1/3]}(\{n\beta\}))$ have  unique Furstenberg systems and they are both   factors of  rotations on the 2-dimensional torus.

We will also use  the following result (again, well known to experts):
\begin{proposition}\label{P:ergodic}
Let $(Y,\nu,S)$ be a system and suppose that   $\nu=\int \nu_y\, d\nu(y)$ is the ergodic decomposition of $\nu$. Then for every $g\in L^\infty(\nu)$
and for almost every $y\in Y$, the sequence $(g(S^ny))$ has a unique  Furstenberg system that is a factor of the system $(Y,\nu_y,S)$ (and as a consequence it is ergodic).
\end{proposition}
\begin{proof}
	Let $g\in L^\infty(\nu)$ be bounded by $1$. 	By the pointwise ergodic theorem for almost every $y\in Y$  the sequence
	$(g(S^ny))$ has a unique Furstenberg system, call it $(X,\mu_y,T)$, where $X=\U^\Z$. Let $\Phi\colon Y\to X$ be
defined by 	
$$
\Phi(y)=(g(S^ny))_{n\in\Z}, \qquad y\in Y.
$$
Then $T\circ \Phi=\Phi\circ S$ and the pointwise ergodic theorem easily implies that for almost every $y\in Y$ we have $\nu_y=\mu_y\circ \Phi^{-1}$. This completes the proof.
	\end{proof}

\section{Background on Hardy fields}\label{S:Hardy}
Let $B$ be the collection of equivalence classes of real valued
functions  defined on some half line $[c,+\infty)$, where we
identify two functions if they agree eventually.\footnote{The
	equivalence classes just defined are often called \emph{germs of
		functions}. We choose to use the word function when we refer to
	elements of $B$ instead, with the understanding that all the
	operations defined and statements made for elements of $B$ are
	considered only for sufficiently large values of $t\in \R$.}
A
\emph{Hardy field} $\mathcal{H}$ is a subfield of the ring $(B,+,\cdot)$ that is
closed under differentiation (the term Hardy field was first used by the
Bourbaki group in \cite{Bour61}). A \emph{Hardy field function} is a
function that belongs to some Hardy field.

\emph{We are going to assume
	throughout that all Hardy
	fields   mentioned are translation invariant},   meaning,
if $a(t)\in \mathcal{H}$,
then  $a(t+h)\in \mathcal{H}$ for every $h\in \R$.


A particular example of such a Hardy field is the set $\mathcal{LE}$ that was
introduced by Hardy in \cite{Ha10} and consists of all
\emph{logarithmic-exponential functions}, meaning all functions
defined on some half line $[c,+\infty)$ by a finite combination of
the symbols $+,-,\times, :, \log, \exp$, operating on the real
variable $t$ and on real constants.
For example,   the functions  $ t^a(\log{t})^b$ where $a,b\in \R$
are all elements of $\mathcal{LE}$.



Every Hardy field function is eventually monotonic and hence has a limit at infinity (possibly infinite). If one of the
functions $a,b\colon [c,+\infty)\to \R$ belongs to a Hardy field and
the other function  belongs to the same Hardy field or to $\mathcal{LE}$,
then the  limit $\lim_{t\to+\infty}a(t)/b(t)$ exists (possibly
infinite). This property is key and  will often justify our use of
l'Hospital's rule. \emph{We are going to freely use all these
	properties without any further explanation in the sequel.} The
reader can find more information about Hardy fields  in \cite{Bo81,
	Bos94} and the references therein.

Recall that $a\colon \R_+\to \R$ has at most  polynomial growth if $a(t)\prec t^d$ for some $d\in \N$.
The most important property of Hardy field functions of at most polynomial growth that will be used throughout this article, is
that we can relate their growth rates with the growth rates of their derivatives.  The next lemma illustrates this principle and will be used frequently:
\begin{lemma}\label{L:derivative}
	Let $a\colon \R_+\to \R$ be a Hardy field function with at most polynomial growth.
	\begin{enumerate}
		\item If  $t^\varepsilon\prec a(t)$ for some $\varepsilon>0$, then
		for every $r\in \N$ we have
	$$
	a'(t)\sim a(t)/t \quad \text{and} \quad a(t+r)-a(t)\sim a(t)/t.
	$$
	
	\item
	If $a(t)\prec t$, then	for every $r\in \N$ we have
	 $$
	\lim_{t\to+\infty} a'(t)=0 \quad \text{and} \quad
	\lim_{t\to+\infty}(a(t+r)-a(t)) =0.
	$$
	\end{enumerate}
\end{lemma}
\begin{proof} We prove $(i)$. Applying l'Hospital's rule (note that all limits below are well defined because $a(t)$ is a Hardy field function) we get
	\begin{equation}\label{E:log}
	\lim_{t\to+\infty}\frac{ta'(t)}{a(t)}=\lim_{t\to+\infty}\frac{(\log{|a(t)|})'}{(\log{t})'}=
	\lim_{t\to+\infty}\frac{\log{|a(t)|}}{\log{t}}.
	\end{equation}
	Since $a(t)$ has at most  polynomial growth and $t^\varepsilon\prec a(t)$,   the last  limit is a positive real number. This proves that 	$a'(t)\sim a(t)/t$. Using the mean value theorem we deduce that  $a(t+r)-a(t)\sim a(t)/t$ for every $r\in \N$.

We prove $(ii)$. Arguing by contradiction suppose that the limit   $\lim_{t\to+\infty} a'(t)$ is  non-zero (the limit exists since $a(t)$ is a Hardy field function). Then an easy application of the mean value theorem gives that the limit $\lim_{t\to+\infty} a(t)/t$ cannot be zero, contradicting our assumption. 
	Finally, using the mean value theorem we deduce  that $\lim_{t\to+\infty}(a(t+r)-a(t)) =0$ for every $r\in\N$.
\end{proof}

We will also use the following equidistribution result:
	\begin{theorem}[Boshernitzan~\cite{Bos94}]\label{T:Boshernitzan}
	Let $a\colon \R_+\to \R$ be a Hardy field function with at most polynomial growth. Then the sequence
	$(a(n))$ is equidistributed on $\T$  	if and only if
	$$
	\lim_{t\to+\infty} \frac{|a(t)-p(t)|}{\log{t}}=+\infty
	$$
	 for every polynomial $p\in\Q[t]$.
\end{theorem}

\section{Proof of results concerning Hardy field sequences}
In this section we will prove the results stated in Section~\ref{SS:HardySequence}.

\subsection{A preliminary result}
We start with a preliminary equidistribution result for Hardy field sequences of sublinear growth. It will be used to define the measure $\lambda$ that appears in the description of the systems ${\bf X}_d$ that are used  in  Theorem~\ref{T:HardySequence}.
	
	\begin{lemma}\label{L:sublinear}
		Let $a\colon \R_+\to \R$ be a Hardy field function such that $a(t)\prec t$. Let $\CP_a$ denote the
		set of probability measures on $\T$ that are weak-star limit points of the sequence of probability measures  $\E_{n\leq N}\,  \delta_{a(n)}$, $N\in \N$. Then the following hold:
			\begin{enumerate}
			\item If $a(t)\succ \log{t}$, then
			$\CP_a=\{m_\T\}$.

			\item If $ a(t)\sim \log{t}$, then
			$\CP_a$ is not a singleton and all its elements are absolutely continuous with respect to $m_\T$.

			\item If $ 1\prec a(t)\prec \log{t}$, then
			$\CP_a=\{\delta_{\alpha}, \alpha\in \T\}$.

			\item If none of the above applies, then $\lim_{t\to+\infty}a(t)=\alpha$, for some $\alpha \in \R$ and
			$\CP_a=\{\delta_{\alpha}\}$. 			
		\end{enumerate}		
		\end{lemma}
	\begin{proof}
		Note that $(i)$ follows from Theorem~\ref{T:Boshernitzan} and that $(iv)$ is trivial.
		
		We prove $(ii)$. We first show that $\mathcal{P}_a$ is not unique.  In this step we will only use  that $1\prec a(t)\ll \log{t}$.  It suffices to show that the
	averages $\E_{n\in[N]}\, e(a(n))$ do not converge as $N\to \infty$. Arguing by contradiction suppose that
	 \begin{equation}\label{E:alpha}
	 \lim_{N\to\infty}\E_{n\in[N]}\, e(a(n))=\alpha
	 \end{equation}
	 for some $\alpha \in \C$.
	Let $\beta \in \S^1$ with $\beta\neq \alpha$. Our assumptions and the mean value theorem imply that   $a(t+1)-a(t)\to 0$ as $t\to+\infty$. Using this and that $|a(t)|\to +\infty$  as $t\to +\infty$ we get  that there exist  $N_k\to+\infty$
	such  that $e(a(N_k))\to \beta$. Moreover, our assumptions and the mean value theorem easily imply  that $\lim_{c\to 1^-}\sup_{s\in [ct,t]}|a(t)-a(s)|=0$, hence
	$$
	\lim_{c\to 1^-}\lim_{k\to\infty} \E_{n\in[cN_k,N_k]}\, e(a(n))=\beta\neq \alpha,
	$$
	which  contradicts \eqref{E:alpha}.

		Next we show that if $\lambda\in \mathcal{P}_a$, then  $\lambda\ll m_\T$. Without loss of generality we can assume that $a(t)\to +\infty$ as $t\to+\infty$ and for simplicity we assume that $a(t)$ (and hence $a^{-1}(t)$)  is strictly increasing on $\R_+$ and $a'(t)$ is positive and strictly decreasing on $\R_+$.
		It suffices to show that there exists a constant $C>0$ that depends only on $a(t)$, such that for every $c,d\in[0,1)$ with $c<d$ we have
		\begin{equation}\label{E:limsup}
		\limsup_{N\to\infty} \frac{|\{n\in [N]\colon \{a(n)\}\in [c,d]\}|}{N}\leq C (d-c).
		\end{equation}
		If we prove this, then   $\lambda\leq C\, m_\T$, hence $\lambda\ll m_\T$.
		
		To this end,  let
		$$
		A_N:=\{n\in [N]\colon \{a(n)\}\in [c,d]\}
		$$
		and note that
		$$
		A_N=\bigcup_{k=0}^{[a(N)]}\{n\in [N]\colon k+c\leq a(n)\leq k+d\}=\\
		 \bigcup_{k=0}^{[a(N)]}\{[a^{-1}(k+c), a^{-1}(k+d)]\cap [N]\}.
		$$
		Since $\big||[a^{-1}(k+c), a^{-1}(k+d)]\cap [N]|-\big(a^{-1}(k+d)-a^{-1}(k+c)\big)\big|\leq 1$ for $k=0,\ldots, [a(N)]-1$, we have
		$$
		|A_N|=\sum_{k=0}^{[a(N)]-1} \big(a^{-1}(k+d)-a^{-1}(k+c)\big)+r_N+ O(a(N))
		$$
		where 
		 $$
		 r_N:=|[a^{-1}([a(N)]+c), a^{-1}(R_N)]\cap [N]| \,
		\text{ and  } \,  R_N:=\min\{[a(N)]+d,a(N)\}.
		$$
		Hence, using the mean value theorem we get that
			\begin{equation}\label{E:AN}
		|A_N|=(d-c)\sum_{k=0}^{[a(N)]-1} (a^{-1})'(\xi_k)+r_N+ O(a(N))
		\end{equation}
		where $\xi_k\in [k+c, k+d]$ for $k=0,\ldots, [a(N)]-1$.

		Next, note that our assumption gives   $a(t)= C_1\log{t}+e_1(t)$ for some $C_1>0$ and $e_1(t)\prec \log{t}$. Since $a(t)$ is a Hardy field function, using l'Hospital's rule we deduce that  $a'(t)=\frac{C_1}{t}+e_2(t)$ where $e_2(t)\prec \frac{1}{t}$. It follows from this that
	 	 \begin{equation}\label{E:e3}
	 (a^{-1})'(t)=\frac{1}{a'(a^{-1}(t))}=C_1^{-1}a^{-1}(t)+e_3(t)
	 \end{equation}
	  where $e_3(t)\prec a^{-1}(t)$.  Hence, for some $C_2>0$ that depends only on $a(t)$, we have
		 $$
		   C_2\cdot (a^{-1})'(\xi_k)\leq a^{-1}(\xi_k)\leq a^{-1}(k+1)
		 $$
		 for $k=0,\ldots,  [a(N)]-1$.
		Using this
		  we get that 
		$$
		 C_2\sum_{k=0}^{[a(N)]-1} (a^{-1})'(\xi_k)\leq \sum_{k=0}^{[a(N)]-1} (a^{-1})(k+1) \leq
		 		 \int_1^{a(N)}a^{-1}(t)\, dt+ N\sim N,
		$$
		where to get the asymptotic for the integral we use l'Hospital's rule, the fundamental theorem of calculus, and that $a'(t)\sim \frac{1}{t}$.
		
		 Finally, we treat the term $r_N$. First note that if  $a(N)\leq [a(N)]+c$, then $r_N=0$. So we can assume that $a(N)> [a(N)]+c$   in which case we have $R_N> [a(N)]+c$. We have that
		$$
		|r_N- \big(a^{-1}(R_N)-a^{-1}([a(N)]+c)\big)|\leq 1,
		$$
		and as before, using the mean value theorem and \eqref{E:e3}, we get that there is a constant $C_3>0$ such that
		$$
		C_3(a^{-1}(R_N)-a^{-1}([a(N)]+c))\leq (R_N-([a(N)]+c)) \,  a^{-1}(a(N))\leq (d-c)N,
		$$
		where we used that $R_N\leq a(N)$ to justify the first estimate and that  $R_N\leq [a(N)]+d$ to justify  the second estimate.
		
		Inserting these estimates in \eqref{E:AN} and using that $a(N)/N\to 0$  as $N\to\infty$,  we deduce that \eqref{E:limsup} holds for some $C>0$.

		We prove $(iii)$. First note that  since $1\prec a(t)\prec t$ and $a(t+1)-a(t)\to 0$  as $t\to+\infty$ (by Lemma~\ref{L:derivative}) we have that the sequence $(a(n))$ is dense in $\T$.
		
Let $\alpha\in[0,1]$. 		 It suffices to show   that if $N_k\to \infty$ is such that $\{a(N_k)\}\to \alpha$, then
		 $$
		 \lim_{k\to\infty}\E_{n\in[N_k]}\delta_{\{a(n)\}}=\delta_\alpha
		 $$
		  where the limit is a weak-star limit.
		Suppose that $\alpha\in (0,1)$ (the argument is similar if $\alpha=0$ or $1$) and $0<c<\alpha<d<1$. It suffices to show that
		\begin{equation}\label{E:lim}
		\lim_{k\to\infty} \frac{|\{n\in [N_k]\colon \{a(n)\}\in [c,d]\}|}{N_k}=1.
		\end{equation}
		
		We first  claim that for every $r\in (0,1)$ we have that
		\begin{equation}\label{E:ct}
		\lim_{t\to +\infty} (a(t)-a(rt))=0.
		\end{equation}
		To see this, notice  first that our assumption $a(t)\prec \log{t}$ and l'Hospital's rule imply that $|a'(t)|$ is eventually decreasing and $a'(t)\prec \frac{1}{t}$. Using this and the mean value theorem, we  deduce that for all large enough $t\in \R$ we have
		$$
		|a(t)-a(rt)|\leq |(1-r)t a'(rt)|\to 0 \quad \text{ as } \, t\to+\infty.
		$$
		
		Since $\{a(N_k)\}\to \alpha\in (c,d)$, it follows from \eqref{E:ct} that for large enough $k\in\N$ we have that
			$$
		[rN_k,N_k]\subset \{n\in [N_k]\colon \{a(n)\}\in [c,d]\}.
		$$
		Hence,
		$$
		\liminf_{k\to\infty} \frac{|\{n\in [N_k]\colon \{a(n)\}\in [c,d]\}|}{N_k}\geq 1-r.
		$$
	 Since $r\in (0,1)$ is arbitrary, letting $r\to 0^+$ we deduce \eqref{E:lim}. This completes the proof.
		\end{proof}

\subsection{Proof of Theorem~\ref{T:HardySequence}}
We are going to deduce Theorem~\ref{T:HardySequence} from the following result:
\begin{proposition}\label{P:fractional2}
		Let $a\colon \R_+\to \R$ be a Hardy field function such that  for some $d\in \Z_+$ one has  $t^d \prec a(t)\prec t^{d+1}$  or $a(t)=t^d\alpha+\tilde{a}(t)$ where $\tilde{a}(t)\prec t^d$ and $\alpha$ is irrational. We define the measure $\lambda$ on $\T$ by  $$
	\lambda:=\lim_{k\to\infty}\E_{n\in [M_k]} \, \delta_{c(n)} \quad \text{where} \quad  c(t):=a^{(d)}(t)/d!,
	$$
	assuming that the previous weak-star limit exists for the sequence $M_k\to+\infty$.  Then the sequence $b(n):=e(a(n))$, $n\in \N$,  admits correlations on $\bM:=([M_k])_{k\in\N}$ and the F-system of $b$ on $\bM$ is isomorphic to the system $(\T^{d+1}, \lambda\times m_{\T^d}, S_d)$ where
	$S_d\colon \T^{d+1}\to \T^{d+1}$ is defined by
	$$
	S_d(y_0,\ldots, y_d)=(y_0,y_1+y_0,\ldots, y_d+y_{d-1}), \quad y_0,\ldots, y_d\in \T.
	$$
	\end{proposition}

Let us first see how we deduce  Theorem~\ref{T:HardySequence} from Proposition~\ref{P:fractional2}.
\begin{proof}[Proof of Theorem~\ref{T:HardySequence} assuming  Proposition~\ref{P:fractional2}]
	First note that since  $\phi\colon \S^1\to  \T$ given by $\phi(e(t)):=t \mod{1}$ is well defined, continuous,  injective, and $\phi(e(a(n)))=a(n)\mod{1}$, $n\in\N$, using  Proposition~\ref{P:ImageRiemann} we get  that it suffices to prove the stated properties for the sequence $b(n):=e(a(n))$, $n\in\N$.
	
	
We move now to the proof of the four parts of the theorem for the sequence $b(n)=e(a(n))$, $n\in\N$.
	Recall that $c(t)=a^{(d)}(t)/d!$, $t\in \R_+$.

We establish Part~$(i)$. First notice that our assumption $t^d\log{t}\prec a(t)\prec t^{d+1}$ and Lemma~\ref{L:derivative} imply that $\log{t}\prec c(t)\prec t$. It follows from Theorem~\ref{T:Boshernitzan}
	that the sequence $(c(n))$ is equidistributed on $\T$. So in this case $\lambda=m_\T$ and  Proposition~\ref{P:fractional2} gives that the sequence $(b(n))$ has a unique F-system that is isomorphic to the system $(\T^d,m_{\T^{d+1}}, S_d)$.
	
We establish	 Parts~$(ii)$ and $(iii)$. Let $(X,\mu,T)$ be an F-system of $b$ on some sequence of intervals $\bM$. Then  Proposition~\ref{P:fractional2} gives that  $(X,\mu,T)$ is isomorphic to the system $(\T^d,\lambda\times m_{\T^d}, S)$.
		Lastly, we show that the sequence $(b(n))$ does not have a unique F-system. For this, it suffices to show that the sequence of measures  $\E_{m\in [M]}\, \delta_{T^mb}$, $M\in \N$, on $(\S^1)^\Z$ does not
		converge weak-star as $M\to \infty$. This would follow if we show that for some $k\in \Z$ the limit
	$$
\lim_{M\to\infty}	\E_{m\in [M]} \,  e(k\Delta^d a(m))
	$$
(which is a correlation of the sequence $b$)	does not exist, where for $a\colon \R_+\to \R$ we let $(\Delta^0 a)(t):=a(t)$, $t\in \R_+$, and for $i\in \Z_+$ we let  $(\Delta^{i+1}a)(t):=a(t+1)-a(t)$, $t\in \R_+$. Note that since $t^d\prec a(t)\ll t^d\log{t}$, Lemma~\ref{L:derivative} gives that  $1\prec \Delta^d a(t)\ll \log{t}$. Hence,  Part~$(iii)$ of Lemma~\ref{L:sublinear} gives that for some $k\in\Z$ the limit
	$\lim_{M\to\infty}\E_{m\in [M]} \,  e(k\Delta^d a(m))$ does not exist.
	
We establish	 Part~$(iv)$.		In this case Lemma~\ref{L:derivative} gives that  $c(t)=a^{(d)}(t)/d!=\alpha/d!+e(t)$, $t\in \R_+$,
	where $e(t):= \tilde{a}^{(d)}(t)\to 0$ as $t\to+\infty$. By Part~$(iv)$ of Lemma~\ref{L:sublinear}  the weak-star limit  $\lim_{M\to\infty}\E_{m\in [M]}\delta_{c(m)}$ exists and is equal to the point mass $\delta_{\frac{\alpha}{d!}}$.
Hence, 	Proposition~\ref{P:fractional2} gives that the sequence $b$ has a unique F-system that is isomorphic
   to the system $(\T^{d+1}, \delta_{\frac{\alpha}{d!}}\times m_{\T^d}, S_d)$. This system is easily shown to be isomorphic to the system $(\T^d,m_{\T^d}, S_d')$, where $S_d'\colon \T^d\to \T^d$ is defined by
	$$
	S'_d(y_1,\ldots, y_d):=(y_1+\alpha/d!,y_2+y_1,\ldots, y_d+y_{d-1}), \quad y_1,\ldots, y_d\in \T.
	$$
	Lastly, since $\alpha$ is irrational, it is well known that this system is totally ergodic.	
	
We establish Part~$(v)$. Suppose that $a(t)$ does not satisfy any  of the conditions in $(i)$-$(iv)$.
	We first claim that    $a(t)=p(t)+\epsilon(t)+\tilde{a}(t)$ for some   $p\in \mathbb{Q}[t]$, $\epsilon\colon \R_+\to\R$ with  $\epsilon(t)\to 0$,  and  $\tilde{a}$ is  a Hardy field function that is covered in cases  $(i)$-$(iv)$. To see this, let $d$ be the largest non-negative integer such that $t^d\ll a(t)$ (then $a(t)\prec t^{d+1}$).
	If $t^d\prec a(t)$, then $a(t)$ is covered
	 by Parts~$(i)$-$(iii)$. 	If $t^d\prec a(t)$ is not satisfied, since  $t^d\ll a(t)$, we have that
	$\lim_{t\to +\infty}a(t)/t^d=:\alpha_d\in \R$. If $\alpha_d$ is irrational, then $a(t)$ is covered by  Part~$(iv)$. If $\alpha_d$ is rational, then $a(t)=t^d\alpha_d+a_1(t)$ where
	$a_1(t):=a(t)-t^d\alpha_d$ satisfies $a_1(t)\prec t^d$. Continuing  like that, we find that there exists $k\in \{0,\ldots, d\}$,  $\alpha_k,\ldots, \alpha_d\in \Q$, $\epsilon\colon \R_+\to \R$,  $\epsilon(t)\to 0$, and  $a_k\colon \R_+\to \R$ such that  $a_k(t)\prec t^k$ and  $a_k(t):=a(t)-(\epsilon(t)+t^k\alpha_k+\cdots+t^d\alpha_d)$ is covered by Parts  $(i)$-$(iv)$ (note that $\epsilon(t)$ is needed only when $k=0$). Then $a(t)=p(t)+\epsilon(t)+\tilde{a}(t)$ with $p(t):=t^k\alpha_k+\cdots+t^d\alpha_d$ and $\tilde{a}(t)=a_k(t)$. Lastly, let  $r$ be the least common multiple of the coefficients of
	$p(t)$ and $k\in \{0,\ldots, r-1\}$. Then $p(rn+k)\in \Z[t]$, hence $b(rn+k)-e(\tilde{a}(rn+k))\to 0$. It follows that the sequences  $(b(rn+k))$ and $(e(\tilde{a}(rn+k)))$ have the same F-systems. Note also that
$a(t)$ satisfies conditions $(i)$-$(iv)$ if and only if the same holds for	$\tilde{a}(rt+k)$.

		
	
	This completes the proof.	
	\end{proof}

Next we move to the proof of  Proposition~\ref{P:fractional2}, which will be based on the following two lemmas. The first one allows us to compute correlations of sequences of the form $(e(a(n)))$ where $a(t)$
is any Hardy field function with at most polynomial growth, and the second one correlations of the sequence $(S_d^nf)$ where $S_d\colon \T^{d+1}\to \T^{d+1}$ is as in Proposition~\ref{P:fractional2} and $f\in C(\T^{d+1})$ is suitably chosen. Our aim is to show that the correlations of the two sequences coincide.

\begin{lemma}\label{L:33}
	Let $a(t)$,   $c(t)$, $\bM$, and $\lambda$
	be as in the statement of Proposition~\ref{P:fractional2}.
	 Then for every $s\in \N$ and $n_1,\ldots, n_s\in \Z$, $k_1,\ldots,k_s\in \{-1,+1\}$, the limit
	$$
	\E_{n\in\bM} \prod_{j=1}^s e(k_ja(n+n_j))
	$$
	exists. Furthermore, if $l_d:=\sum_{j=1}^sk_j   n_j^d$, then this limit  is equal to $\int e(l_d\, t)\, d\lambda(t)$ if $\sum_{j=1}^s k_jn_j^i=0$ for $i=0,\ldots, d-1$, and is equal to $0$ otherwise.
\end{lemma}
\begin{proof}
	By our assumptions and Lemma~\ref{L:derivative}
	we have that $\lim_{t\to+\infty}a^{(d+1)}(t)=0$. Using this and Taylor expansion, we get that
	for every $h\in \Z$ we have
	$$
	a(t+h)=\sum_{i=0}^da^{(i)}(t) \frac{h^i}{i!}+\epsilon_h(t), \quad t\in \R_+,
	$$
	where the function $\epsilon_h\colon \R_+\to \R$ satisfies $\lim_{t\to +\infty}\epsilon_h(t)=0$.
	So if  $$
	A(t):=\sum_{j=1}^s k_j a(t+n_j), \quad  t\in \R_+,
	$$
	we have that
	$$
	 \prod_{j=1}^s e(k_ja(n+n_j))= e(A(n)), \quad n\in\N,
	$$
	and
	$$
	A(t)=\sum_{i=0}^d c_i a^{(i)}(t)+\epsilon(t), \quad t\in \R_+,
	$$
(note that $A(t)$ is again a Hardy field function)	where
	$$
	c_i:= \frac{1}{i!}\sum_{j=1}^sk_j  n_j^i, \quad  i=0,\ldots, d,  \qquad \epsilon(t):=\sum_{j=1}^sk_j\epsilon_{n_j}(t), \quad t\in \R_+.
	$$
	Note that $\lim_{t\to +\infty}\epsilon(t)=0$.  Recall that  $c:=a^{(d)}/d!$. If $c_0=\cdots=c_{d-1}=0$,  then $\lim_{t\to +\infty}(A(t)-l_d\, c(t))=0$, where $l_d:=\sum_{j=1}^sk_j  n_j^d$. Hence,
	$$
	\E_{n\in\bM}\, e(A(n))=\E_{n\in\bM}\, e(l_d\, c(n))=\int e(l_d\, t)\, d\lambda(t).
	$$
	 Otherwise, let $i_0$ be the smallest $i\in \{0,\ldots, d-1\}$ such that $c_i\neq 0$. Using Lemma~\ref{L:derivative}  and our assumptions on $a(t)$, we deduce that either  $t^{d-i_0}\prec A(t)\prec t^{d-i_0+1}$ or  $|A(t)-t^{d-i_0}\beta| \prec t^{d-i_0}$
	 where $\beta:=i_0!\alpha$ is irrational. 
 Since $d-i_0\geq 1$, in both cases we have by Theorem~\ref{T:Boshernitzan} that  $\E_{n\in\N}\, e(A(n))=0$.
	This completes the proof.
\end{proof}

In the statement below we use the convention $0^0=1$.
\begin{lemma}\label{L:34}
	Let $a(t)$,   $c(t)$, $\bM$,  $\lambda$, $d$,  and $S_d$
	be as in the statement of Proposition~\ref{P:fractional2}.
Let also $f\colon \T^{d+1}\to \S^1$ be defined by $f(y):=e(y_d)$ for $y=(y_0,\ldots, y_d)\in \T^{d+1}$. Then  for every $s\in \N$ and $n_1,\ldots, n_s\in \Z$, $k_1,\ldots,k_s\in \{-1,+1\}$, we have that  the integral
	\begin{equation}\label{E:integral}
	\int \prod_{j=1}^s S_d^{n_j}f^{k_j} \, d(\lambda \times m_{\T^d})
	\end{equation}
  is equal to $\int e(l_d\, t)\, d\lambda(t)$ if $\sum_{j=1}^s k_jn_j^i=0$ for $i=0,\ldots, d-1$,
  where $l_d:=\sum_{j=1}^sk_j   n_j^d$, and is equal to $0$ otherwise.
	\end{lemma}
\begin{proof}
For $y=(y_0,\ldots, y_d)\in \T^{d+1}$ direct computation gives that
	$$
		f(S_d^ny)= e\Big(\sum_{i=0}^d\binom{n}{i} y_{d-i}\Big), \qquad n\in\N.
	$$
	Hence, for $y\in \T^{d+1}$ we have
$$
\prod_{j=1}^s f^{k_j}(S_d^{n_j}y)=e\Big(\sum_{i=0}^d c_i y_{d-i}\Big)
$$
where
	$$
c_i:=\sum_{j=1}^s k_j \binom{n_j}{i}, \quad i=0,\ldots, d.
	$$
	
It follows that  the integral in \eqref{E:integral}  is
	 equal to  $\int e(c_d y_0)\, d\lambda(y)$  if $c_i=0$ for $i=0,\ldots, d-1$,
	 and is equal to  $0$ otherwise. 	
	 Lastly, one easily verifies  that $c_i=0$ for $i=0,\ldots, d-1$ if and only if
	 $\sum_{j=1}^s k_jn_j^i=0$ for $i=0,\ldots, d-1$, which implies that  $c_d=\sum_{j=1}^sk_jn_j^d=l_d$.
	 	This completes the proof.
	\end{proof}
Combining the previous two results we can now prove Proposition~\ref{P:fractional2}.
\begin{proof}[Proof of Proposition~\ref{P:fractional2}]	
	By the first part of the statement of Lemma~\ref{L:33} we get that the sequence $b$ admits correlations on $\bM$. Let $(X,\mu,T)$ be the F-system of $b$ on $\bM$
	where, as usual,  $X=(\S^1)^{\Z}$,
	$T$ is the shift transformation on $X$, and $\mu:=\lim_{k\to\infty}\E_{m\in [M_k]}\, \delta_{T^mb}$. It remains to establish the asserted isomorphism. To this end, we define  the map $\Phi\colon \T^{d+1}\to X$  by
	$$
	\Phi(y):=(f(S_d^ny))_{n\in \Z}, \qquad y\in \T^{d+1},
	$$
	where $y=(y_0,\ldots, y_d)$ and $f(y):=e(y_d)$, $y\in \T^{d+1}$. We clearly have $\Phi\circ S_d=T\circ \Phi$. Moreoever, it is easy to check that the map $\Phi$ is injective. It remains to verify that $\mu=(\lambda\times m_{\T^d})\circ\Phi^{-1}$. Let $F_0(x):=x(0 )$, $x\in X$. For every $s\in \N$ and $n_1,\ldots, n_s\in \Z$,  $k_1,\ldots, k_s\in \{-1,+1\}$,
	we have
	$$
	\int \prod_{j=1}^sT^{n_j}F^{k_j}_{0}\, d\mu=\E_{n\in\bM}  \prod_{j=1}^se(k_ja(n+n_j))=\int \prod_{j=1}^s S_d^{n_j}f^{k_j} \, d(\lambda\times m_{\T^{d+1}}),
	$$
	where the first identity follows from \eqref{E:correspondence} and the second identity follows by combining Lemma~\ref{L:33} and Lemma~\ref{L:34}. Using this and the fact that  $f=F_0\circ \Phi$, we get that a linearly dense subset of $C(X)$ has the same integral with respect to the  measures $\mu$ and $(\lambda \times m_{\T^d})\circ\Phi^{-1}$, hence, the two measures  coincide. This completes the proof.
\end{proof}

\subsection{Proof of Corollaries~\ref{C:disjoint}, \ref{C:13}, \ref{C:13'}}

 We start with the proof of Corollary~\ref{C:disjoint}, which is a consequence of
 the structural result of Theorem~\ref{T:HardySequence} and an ergodic theorem from
 \cite{L0}.

 \begin{proof}[Proof of Corollary~\ref{C:disjoint}]
Let $(X,\mu,T)$ be an F-system of $b$. 	By Parts~$(i)$ and $(ii)$  of Theorem~\ref{T:HardySequence} the system $(X,\mu,T)$ is isomorphic to the system
${\bf X}_d:=(\T^{d+1}, \nu:=\lambda\times m_{\T^{d}}, S_d)$ where $\lambda$ is a continuous probability measure on $\T$ and  $S_d$ is the unipotent homomorphism of $\T^{d+1}$  defined by
$$
S_d(y_0,\ldots, y_d):=(y_0,y_1+y_0,\ldots, y_d+y_{d-1}), \quad y_0,\ldots, y_d\in \T.
$$
Hence,  it remains to show that ${\bf X}_d$  is disjoint from every ergodic system  $(Z,\rho, R)$.
So let $\sigma$ be a joining of these two systems. In order to show that
$\sigma=\nu\times \rho$ it suffices to show that for  every $f\in C(\T^{d+1})$
with $\int f \, d\nu=0$ and every $g\in L^\infty(\rho)$ we have
$$
\int f(y) \, g(z) \, d\sigma(y,z)= 0.
$$

Since $\sigma$ is $(T\times R)$-invariant it suffices to show that
\begin{equation}\label{E:yz}
\lim_{N\to\infty}\E_{n\in[N]}\int f(S_d^ny) \, g(R^nz) \, d\sigma(y,z)= 0.
\end{equation}
Using uniform approximation of  $f$ by trigonometric polynomials we can assume that
$f$ is a complex exponential of $\T^{d+1}$.

Let $y:=(y_0,\ldots, y_d)$ and suppose first that
 $f(y)=e(ky_0)$ for some $k\in \Z$.
Then $f(S_d^ny)=e(ky_0)$ for every $n\in\N$ and using the ergodicity of the system $(Z,\rho, R)$ we get that the limit in \eqref{E:yz} is equal to $\int f\, d\mu \cdot \int g\, d\nu=0$. So we can assume that
 there exists  $d'\in \{1,\ldots, d\}$ such  that
$f(y)=e(\sum_{k=0}^{d'} l_ky_k)$ for some $l_0,\ldots, l_{d'}\in \Z$ with
  $l_{d'}\neq 0$.
 In this case,  a simple computation gives that
 $$
 f(S_d^ny)=e\big(q y_0n^{d'}+p_y(n)\big), \qquad y\in \T^{d+1},
 $$
where $q:=l_{d'}/d'!$ and for every $y\in \T^{d+1}$ we have that $p_y$ is a polynomial with real coefficients and degree strictly smaller than $d'$.

Using again the ergodicity of the system $(Z,\rho,R)$ and  \cite[Theorem~4]{L0},
we get that there exists a subset $Z'$ of $Z$ with $\rho(Z')=1$
such that the following holds: For
every $\alpha\in \R$ such that $e(k\alpha)$  is   not an eigenvalue of $(Z,\rho,R)$    for every non-zero $k\in \Z$, we have
 $$
 \lim_{N\to\infty }\E_{n\in [N]}\, e(\alpha n^{d'}+p(n))\, g(R^nz)=0
 $$
  for every polynomial $p$ of degree smaller than $d'$ and every $z\in Z'$.
Let $Y'$ be the set of all $y=(y_0,\ldots, y_d)\in \T^{d+1}$ such that
$e(ky_0)$ is  not an eigenvalue of $(Z,\rho,R)$ for every non-zero $k\in \Z$.
Obviously the projection of  $Y'$ on the $y_0$ coordinate differs from $\T$ on a countable set. Since   $\nu=\lambda\times m_{\T^d}$ and the measure $\lambda$  is continuous, we have that $\nu(Y')=1$. From the above we get that for all
$(y,z)\in Y'\times Z'$ we have
$$
\lim_{N\to\infty}\E_{n\in[N]} f(S_d^ny) \, g(R^nz) = 0.
$$
Since $\rho(Y'\times Z')=1$, the bounded convergence theorem implies that \eqref{E:yz} holds.
This  completes the proof.
 	\end{proof}

Corollary~\ref{C:13} is a consequence of  Corollary~\ref{C:disjoint}
and some pretty standard maneuvers that enables us  to pass orthogonality statements
from the sequences $(e(a(n)t))$, $t\in\R\setminus\{0\}$, to the sequence
$(e([a(n)]\alpha))$ where $\alpha \in \R\setminus \Z$.

\begin{proof}[Proof of Corollary~\ref{C:13}]
	Suppose first that $b(n)=e(a(n))$, $n\in\N$.
	Arguing by contradiction,  suppose that \eqref{E:w} fails for some ergodic sequence $w\colon \N\to \U$. Then there exists a sequence of intervals $\bM=([M_k])_{k\in\N}$, with $M_k\to \infty$, such that the
	sequences $b,w$ admit joint correlations on $\bM$ and
	\begin{equation}\label{E:neq0}
	\E_{n\in \bM}\,  b(n) \, w(n) \neq 0.
	\end{equation}
	By Corollary~\ref{C:disjoint} the F-systems of $b,w$ on $\bM$ are disjoint, hence their joint F-system
	(which is a joining of the two  systems) is the direct product of these systems. This easily implies that
	$$
	\E_{n\in \bM}\,  b(n) \, w(n)=\E_{n\in \bM}\,  b(n) \cdot \E_{n\in \bM}\,   w(n)=0,
	$$
	where the last equality holds since $\E_{n\in\N}\, e(a(n))=0$ by Theorem~\ref{T:Boshernitzan}. This contradicts \eqref{E:neq0} and completes the proof in the case where $b(n)=e(a(n))$, $n\in\N$.
	
	Suppose now that $b(n)=e([a(n)]\alpha)$, $n\in\N$,  for some $\alpha\in \R\setminus \Z$. First note that
	$b(n)= e(a(n)\alpha) \phi(a(n))$, $n\in \N$,  where $\phi\colon \T\to \S^1$ is given by $\phi(t):=e(-\{t\}\alpha)$, $t\in \T$, is Riemann-integrable with respect to the measure $m_\T$. Hence, it suffices to show that  under our assumptions on the sequences $a,w$, for every  $\alpha\in \R\setminus \Z$ we have that
	\begin{equation}\label{E:aw0}
	\E_{n\in \N}\,  e(a(n)\alpha)\,  \phi(a(n)) \, w(n) =0
	\end{equation}
	for every $\phi\colon \T\to \C$ that is Riemann-integrable (with respect to $m_\T$).
	Suppose first that $\phi(t):=e(kt)$, $t\in \T$,  for some $k\in \Z$. Then
	$e(a(n)\alpha)\,  \phi(a(n))= e(a(n)(\alpha+k))$,  and since by assumption $\alpha+k\neq 0$, we get by the previous case (for $a(n)(\alpha+k)$ in place of $a(n)$) that
	\eqref{E:aw0} holds. Using linearity and uniform approximation by trigonometric polynomials, we deduce that \eqref{E:aw0} also holds when $\phi\in C(\T)$. Finally, let $\varepsilon>0$ and $\phi_\varepsilon\in C(\T)$ be such that $\norm{\phi-\phi_\varepsilon}_{L^1(m_\T)}\leq \varepsilon$. Then using that \eqref{E:aw0} holds for $\phi_\varepsilon$ in place of $\phi$ we get that
	\begin{multline*}
	\limsup_{N\to\infty}|\E_{n\in [N]}\,  e(a(n)\alpha)\,  \phi(a(n)) \, w(n)| \ll  \\
	\limsup_{N\to\infty} \E_{n\in [N]}\, |\phi-\phi_\varepsilon|(a(n)) = \norm{\phi-\phi_\varepsilon}_{L^1(m_\T)}\leq \varepsilon,
	\end{multline*}
	where to justify the last identity we used that the sequence $(a(n))$ is equidistributed on $\T$ by Theorem~\ref{T:Boshernitzan}.
	Since $\varepsilon$ is arbitrary we get that \eqref{E:aw0} holds, completing the proof.
\end{proof}
 Corollary~\ref{C:13'} is a simple consequence of Corollary~\ref{C:13}.
\begin{proof}[Proof of Corollary~\ref{C:13'}]
	If $E$ is the range of the sequence $b$, then for every $k\in \Z$ one easily verifies that
	$$
	d(E)\cdot \E_{n\in\N}\, e(ka(b(n)))=\E_{n\in\N}\big(e(ka(n))\cdot {\bf 1}_E(n)\big).
	$$
	By assumption the sequence ${\bf 1}_E(n)$ is ergodic, hence if $k\neq 0$ we have that Corollary~\ref{C:13} applies and gives that  the last average is $0$. Since $d(E)>0$ we deduce that
	$\E_{n\in\N}\, e(ka(b(n)))=0$ for every non-zero $k\in \Z$. Hence,  the sequence $\big((a\circ b)(n)\big)_{n\in\N}$ is equidistributed on $\T$.
	
	Similarly, one verifies that for every $k\in\Z$ and $\alpha\in \R$ such that $k\alpha\not\in \Z$ we have
	$$
	\E_{n\in\N}\, e(k[a(b(n))]\alpha)=0.
	$$
	This implies the other two equidistribution properties and completes the proof.
\end{proof}

\subsection{Proof of Theorem~\ref{T:HardySequence'}}
The proof of Theorem~\ref{T:HardySequence'} will be based on the next result:
\begin{proposition}\label{P:fractional3}
	Let  $a\colon \R_+\to \R$ be  a Hardy field function such that  $t^d\log{t} \prec a(t)\prec t^{d+1}$ for some $d\in \Z_+$. Then for every $\alpha\in \R\setminus \Q$ the pair of sequences
	 $(a(n)), (a(n) \alpha)$ (with elements on $\T$) has a unique joint F-system that  is  isomorphic to the system $(\T^{2(d+1)},m_{\T^{2(d+1)}}, S_d\times S_d)$ where
	$S_d$ is given by \eqref{E:unipotent}.
\end{proposition}
\begin{proof}
	By Proposition~\ref{P:ImageRiemann} the F-system of the sequences $(a(n))$ (with elements on $\T$) and $(e(a(n)))$ are isomorphic. Moreover, a similar argument gives that a joint F-system on $\bM$  of the
	pair of sequences $(a(n)), (a(n)\alpha)$ and the corresponding one for the pair of sequences $(e(a(n))), (e(a(n)\alpha))$ are isomorphic. Hence, it suffices to establish the asserted statement with the sequences
	$(e(a(n)))$ and  $(e(a(n)\alpha))$  in place of the sequences $(a(n))$ and  $(a(n)\alpha)$ respectively.
	
	Let $(X,\mu,T)$ be the F-system of the sequence  $(e(a(n)))$ and $(X\times X,\nu,S) $, where $S=T\times T$,  be a joint  F-system of
	the pair of sequences $(e(a(n))), (e(a(n) \alpha))$ on $\bM$.
	It suffices to show that
	$\nu=\mu\times \mu$.
	For $i=1,2$ let $F_{i,0}\in C(X\times X)$ be defined by $F_{i,0}(x_1,x_2):=x_i(0)$, where $(x_1,x_2)\in X\times X$.  Since the collection of functions of the form   $\prod_{i=1}^2 \prod_{j=1}^sS^{n_{i,j}}F_{i,0}^{k_{i,j}}$, where  $n_{i,j}\in \Z$, $k_{i,j}\in \{-1,+1\}$, for  $i\in \{1,2\}$, $j\in \{1,\ldots, s\}$, $s\in \N$,   is linearly dense in $C(X\times X)$,  it suffices to show that
	$$
	\int \prod_{i=1}^2 \prod_{j=1}^sx_i^{k_{i,j}}(n_{i,j}) \, d\nu=
	\prod_{i=1}^2 \int \prod_{j=1}^sx_i^{k_{i,j}}(n_{i,j}) \, d\mu
	$$
	for  every $s\in \N$ and  $n_{i,j}\in \Z$, $k_{i,j}\in \{-1,+1\}$,  $i\in \{1,2\}$, $j\in\{1,\ldots, s\}$.

	For notational  convenience, we let $a_1:=a$ and $a_2:=\alpha \cdot a$. Using the definition of the measures $\mu$ and $\nu$  we get that  it suffices to show that for every 	  $s\in \N$ and  $n_{i,j}\in \Z$, $k_{i,j}\in \{-1,+1\}$,  where  $i\in \{1,2\}$, $j\in\{1,\ldots, s\}$,
	we have the identity
	\begin{equation}\label{E:identityij}
	\E_{n\in\bM}  \prod_{i=1}^2\prod_{j=1}^se(k_{i,j}a_i(n+n_{i,j}))=
	\prod_{i=1}^2\Big(\E_{n\in\bM}  \prod_{j=1}^se(k_{i,j}a_i(n+n_{i,j}))\Big).
	\end{equation}
	
	By assumption we have that  $t^{d}\log{t}\prec a(t)\prec t^{d+1}$ for some $d\in \Z_+$.
	By Lemma~\ref{L:33} the right hand side in \eqref{E:identityij} is $1$ if
	\begin{equation}\label{E:iszero}
	\sum_{j=1}^sk_{i,j}n_{i,j}^r=0 \ \text{ for } \  i=1, 2 \ \text{ and } \  r=0,\ldots, d,
	\end{equation}
	and is $0$ otherwise.
	
 Next we deal with the left hand side in \eqref{E:identityij}. Using Taylor expansion and arguing exactly as in the proof of
	Lemma~\ref{L:33} we find that  if  $$
	A(t):=\sum_{i=1}^{2}\sum_{j=1}^s k_{i,j} a_i(t+n_{i,j}), \quad  t\in \R_+,
	$$
	then for some  $e\colon \R_+\to \R$ that satisfies $\lim_{t\to+\infty}e(t)=0$ we have
	$$
	A(t)=\sum_{i=1}^2 \sum_{r=0}^{d} c_{i,r} a_i^{(r)}(t)+\epsilon(t)=\sum_{r=0}^{d} (c_{1,r}+c_{2,r}\alpha) a^{(r)}(t)+\epsilon(t),
	$$
	where for $i=1, 2$ we have
	$$
	c_{i,r}:= \frac{1}{r!}\sum_{j=1}^sk_{i,j}  n_{i,j}^r, \quad r=0,\ldots, d.
	$$
	We deduce that if \eqref{E:iszero} holds,  then $A(t)=\epsilon(t)\to 0$ as $t\to+\infty$. Therefore, we have   $\E_{n\in \N}\, e(A(n))=1$ and the left hand side in \eqref{E:identityij} is $1$.  Suppose now that \eqref{E:iszero} does not hold.  Then
	$\sum_{j=1}^sk_{i,j}n_{i,j}^r\neq 0$ for some $i\in \{1, 2\}$ and $r\in \{0,\ldots, d\}$.
	Let $r_0$
	be the smallest $r\in \{0,\ldots, d\}$ such that $|c_{1,r}|+|c_{2,r}|\neq 0$. Since $c_{1,r},c_{2,r}$ are rational and $\alpha$ is irrational, we have that $c_{1,r}+c_{2,r}\alpha\neq 0$.  Using  Lemma~\ref{L:derivative} we get that  $A\sim a^{(r_0)}$ and  deduce that $A(t)\sim a(t)/t^{r_0}$ for some $r\in \{0,\ldots, d\}$. Combining this with  Theorem~\ref{T:Boshernitzan} we get that
	$\E_{n\in\N}\, e(A(n))=0$. We  deduce that in all cases \eqref{E:identityij} holds.  This completes the proof.
\end{proof}
We can now proceed to the proof of  Theorem~\ref{T:HardySequence'}.
\begin{proof}[Proof of Theorem~\ref{T:HardySequence'}]
	First  we carry out a reduction. Let  $c(n):=[a(n)]\alpha$, $n\in\N$.
	Theorem~\ref{T:Boshernitzan}  gives that for every non-zero $t\in \R$  the sequence $(a(n)t)$ is equidistributed on $\T$,
	and using a standard argument (see for example the proof of \cite[Theorem~6.3]{BKQW05}) we deduce that
	for every irrational $\alpha$  the sequence $(c(n))$ is equidistributed on $\T$. Since $b(n)=\phi(c(n))$, $n\in\N$, and the
	function  $\phi$ is Riemann-integrable (with respect to $m_\T$),   Proposition~\ref{P:ImageRiemann} applies
	and gives that in order to get the asserted properties for the sequence $b$ it suffices to get them for the sequence $c$.

	We thus turn our attention to the sequence $c$.
	Note first that if
	$\psi\colon \T^{2}\to \T$ is defined by
	$$
	\psi(x,y):=y-\{x\}\alpha \mod{1},
	$$
	then $\psi$ is Riemann-integrable with respect to $m_{\T^2}$ (the set of discontinuities of $\psi$ has $m_{\T^2}$-measure $0$) and we have the identity
	$$
	c(n)=\psi(a(n), a(n)\alpha), \quad n\in \N.
	$$
	Next, note that
	by Part~$(i)$ of Theorem~\ref{T:HardySequence}	for every non-zero $\alpha\in \R$ 	the sequences  $(a(n))$ and   $(a(n)\alpha )$ on $\T$ have  unique F-systems, and they are both isomorphic 	to
	the  system $(\T^{d+1}, m_{\T^{d+1}}, S_d)$ where
		$S_d\colon \T^{d+1}\to \T^{d+1}$ is given by \eqref{E:unipotent}.
	By Proposition~\ref{P:fractional3} for every irrational  $\alpha \in \R$	the pair  of sequences $(a(n)), (a(n) \alpha)$ has a unique joint F-system,
	and  it  is isomorphic to the system $(\T^{2(d+1)},m_{\T^{2(d+1)}}, S_d\times S_d)$.
	The needed conclusion for the sequence $c$ now follows from Proposition~\ref{P:ImageRiemann}, assuming that we verify that the sequence
	$(a(n), a(n)\alpha)$ is equidistributed on $\T^2$ with respect
	to the Haar measure $m_{\T^2}$. Since $\alpha$ is irrational,  this easily follows from Theorem~\ref{T:Boshernitzan}, our assumption $t^d\log{t} \prec a(t)\prec t^{d+1}$ for some $d\in \Z_+$, and Weyl's equidistribution theorem. This completes the proof. 	 	
\end{proof}

\section{Proof of results concerning Hardy field iterates}
The proof of Theorem~\ref{T:HardyIterates} is a direct consequence of
the next result, which establishes strong stationarity for the Furstenberg systems defined in Theorem~\ref{T:HardyIterates}, and Theorem~\ref{T:Sst} that describes the structure of strongly stationary systems.
\begin{theorem}\label{T:SstHardyIterates}
Let $a\colon \R_+\to \R$ be a  Hardy field function such that $t^{d+\varepsilon}\prec a(t)\prec t^{d+1}$ for some $\varepsilon>0$. Furthermore, let  $(Y,\nu,S)$ be a  system.
	Then every strictly increasing sequence $(N_k)$ has a subsequence $(N_k')$
 such that for almost every $y\in Y$ and for every $g\in L^\infty(\nu)$ the sequence $(g(S^{[a(n)]}y))$ admits correlations  on $\bN':=([N'_k])_{k\in\N}$ and the corresponding Furstenberg system is strongly stationary.
\end{theorem}

The remainder of this section is devoted to the proof of   Theorem~\ref{T:SstHardyIterates}.

\subsection{Strong stationarity of Hardy-field nilsequences}\label{SS:sstnil}
If $G$ is a group  we let
$G_1:=G$ and $G_{j+1}:=[G,G_j]$, $j\in \mathbb{N}$. We say that $G$ is {\em nilpotent} if  $G_{s}$ is the trivial
group for some $s\in \N$.
A {\em  nilmanifold} is a homogeneous space $X=G/\Gamma$, where  $G$ is a  nilpotent Lie group
and $\Gamma$ is a discrete cocompact subgroup of $G$. With $e_X$ we denote  the image in $X$ of the unit element of $G$.
A {\it   nilsystem} is
a system of the form $(X, \CX, m_X, T_b)$,  where $X=G/\Gamma$ is a nilmanifold, $b\in G$,  $T_b\colon X\to X$ is defined by  $T_b(g\cdot e_X) \mathrel{\mathop:}= (bg)\cdot e_X$ for  $g\in G$,  $m_X$  is
the normalized Haar measure on $X$, and  $\CX$  is the completion
of the Borel $\sigma$-algebra of $G/\Gamma$. If $G$ is connected and simply connected and $b\in G$, then $b^t$ is well defined  for every $t\in \R$.

 The first step in the proof of  Theorem~\ref{T:HardyIterates} is  to establish strong stationarity in the case where the system $(Y,\nu,S)$ is a nilsystem. Although in the proof of   Theorem~\ref{T:HardyIterates} we only use Proposition~\ref{P:NilHardyFormula'},  we state and prove the following result that is of independent interest:
\begin{theorem}\label{T:NilHardyIterates}
	Let $a\colon \R_+\to \R$ be a  Hardy field function such that $t^{d+\varepsilon}\prec a(t)\prec t^{d+1}$ for some $d\in \Z_+$ and $\varepsilon>0$. Let $X=G/\Gamma$ be a nilmanifold and $b\in G$. Then for every $x\in X$ the sequences $(b^{a(n)}x)$ and $(b^{[a(n)]}x)$ have unique F-systems that are strongly stationary (in the first case we assume that $G$ is connected and simply connected).
\end{theorem}
\begin{remark}
It can be shown that the symbolic systems of the above  sequences have zero topological entropy. If we combine this with Theorem~\ref{T:Sst},
we get  that the ergodic components of the corresponding F-systems are infinite-step nilsystems (we expect but does not follow from our arguments that they are finite-step nilsystems). It would be interesting to verify  that a similar property holds for all Hardy field functions $a(t)$ with at most polynomial growth.
\end{remark}
The proof of Theorem~\ref{T:NilHardyIterates} will be based on an equidistribution result
from \cite{BMR20} that was proved for certain weighted averages that we define next.
For $r\in \N$, let
$$
\Delta^1_r a:=a(n+r)-a(n), \quad n\in\N,
$$
 and  for $i\in \N$ define inductively
 $\Delta_r^{i+1}a:=\Delta^1_r(\Delta_r^i a)$.


If  $w\colon \N\to \R_+$ is an eventually  increasing sequence  and $\lim_{n\to\infty}w(n)= +\infty$,  then for every $a\colon \N\to \U$ we
let (for those $N\in \N$ for which  $w(N)\neq 0$)
$$
\wE_{n\in [N]}\, a(n):=\frac{1}{w(N)}\sum_{n=1}^N(w(n+1)-w(n))\, a(n).
$$
For example if $w(n)=n, n\in\N$, then we get the Ces\`aro averages, and if $w(n)=\log{n}, n\in\N$, then we get an averaging scheme equivalent to logarithmic averages.
The next result is a direct consequence of results proved in~\cite{BMR20}.

 \begin{proposition}[\cite{BMR20}]\label{P:NilHardyFormula}
	Let $k,r\in \N$, $d\in \Z^+$, and  $a\colon \R_+\to \R$ be a  Hardy field function such that $t^{d+\varepsilon}\prec a(t)\prec t^{d+1}$ for some $d\in \Z_+$ and $\varepsilon>0$. For $j=1,\ldots,k$ let $a_{j,r}:=\sum_{i=0}^d c_{i,j}\Delta_r^ia$ for some $c_{0,j},\ldots, c_{d,j}\in \R$.
 Let $X=G/\Gamma$ be a nilmanifold and $b\in G$. Then for every $x\in X$ and $h_1,\ldots, h_k\in C(X)$ the limits
 	$$
	\lim_{N\to\infty}\wrE_{n\in [N]}\prod_{j=1}^k h_j(b^{a_{j,r}(n)}\cdot x), \qquad  	\lim_{N\to\infty}\wrE_{n\in [N]} \prod_{j=1}^k h_j(b^{[a_{j,r}(n)]}\cdot x)
 	$$
 	exist and do not depend on $r$, where $w_r:=|\Delta^d_ra|$  (in the first case we assume that $G$ is connected and simply connected).
 \end{proposition}
\begin{proof}
For $x:=e_X$, 	it is proved in \cite[Theorem~4.12 and Corollary~4.13]{BMR20}  (for $r=1$ but the same argument works for general $r\in \N$) that  the two limits exist and it follows from the proof that  the limit  does not depend on  $r\in \N$.
For general $x\in X$, one writes $x=g\cdot e_X$ for some $g\in G$
and applies the previous result for $b':=g^{-1}bg$ and $h_j'(x):=h_j(gx)$, $x\in X$, for $j=1,\ldots, k$.
	\end{proof}
The next lemma enables us to deduce from Proposition~\ref{P:NilHardyFormula} a similar result for Ces\`aro averages.

\begin{lemma}\label{L:Wr}
	Let $r\in \N$ and $d\in \Z^+$. Let $a\colon \R_+\to \R_+$ be a  Hardy field function such that $t^{d+\varepsilon} \prec a(t)\prec t^{d+1}$ for some $d\in \Z_+$ and $\varepsilon>0$, and let $w:=|\Delta_r^da|$. If $(X,\norm{\cdot})$ is a normed space and   $b\colon \N\to X$ is a bounded sequence such that
	$$
	\lim_{N\to\infty} \wE_{n\in [N]}\,  b(n)=L,
	$$
	then
	$$
	\lim_{N\to\infty}\E_{n\in[N]}\, b(n)=L.
	$$
\end{lemma}
\begin{proof}
	Note that by Lemma~\ref{L:derivative} we have that
	$t^\varepsilon \prec w(t)\prec t$ and $w$ is eventually increasing. Hence,
	$$\lim_{t\to +\infty} \log(w(t))/\log{t}\neq 0
	$$ and   using l'Hospital's rule we get
	that  $\lim_{t\to +\infty} tw'(t)/w(t)\neq 0$ (all limits exist since $a$ is a Hardy field function). Using the mean value theorem
	twice, that $w''$ is eventually monotonic,  and that $w''(t)\prec w'(t)$, we get that  $\lim_{t\to +\infty}(w(t+1)-w(t))/w'(t)=1$.
We deduce that  $$
	\lim_{t\to +\infty} t(w(t+1)-w(t))/w(t)\neq 0.
	$$ Hence, if we let  $u(t):= w(t+1)-w(t)$, $t\in \R_+$, we have that the sequence $U(n)/(nu(n))$ is bounded, where $U(n):=u(1)+\cdots +u(n)$, $n\in\N$.

		Therefore, we have reduced matters to proving the following elementary statement: Let
		$(X,\norm{\cdot})$ be a normed space and $b\colon \N\to X$ be a bounded sequence.
Let also 		$u\colon \N\to \R_+$ be eventually decreasing,
	$U(n)/(nu(n))$ be  bounded, where $U(n):=u(1)+\cdots+u(n)$, $n\in\N$, and  suppose  that
	$$
	\lim_{N\to\infty}\frac{1}{U(N)} \sum_{n=1}^N\, u(n)b(n)=L.
	$$
	Then
	$$
	\lim_{N\to\infty}\E_{n\in[N]} \, b(n)=L.
	$$
	This is a straightforward  exercise in partial summation. 	
	\end{proof}

Combining the previous two results we get the following:
	\begin{proposition}\label{P:NilHardyFormula'}
	Let $k,r\in \N$, $d\in \Z^+$, and  $a\colon \R_+\to \R$ be a  Hardy field function such that $t^{d+\varepsilon}\prec a(t)\prec t^{d+1}$ for some $d\in \Z_+$ and $\varepsilon>0$.
	Let also $X=G/\Gamma$ be a nilmanifold,  $b\in G$, and $h_1,\ldots, h_k\in C(X)$. Then for every $x\in X$  and $n_1,\ldots, n_k\in \N$ the limits
	\begin{equation}\label{E:identities}
	\lim_{N\to\infty}\E_{n\in[N]}\prod_{j=1}^k h_j(b^{a(n+rn_j)}\cdot x),\qquad  		\lim_{N\to\infty}\E_{n\in[N]}\  \prod_{j=1}^k h_j(b^{[a(n+rn_j)]}\cdot x)
	\end{equation}
	exist and do not depend on $r$  (in the first case we assume that $G$ is connected and simply connected).
\end{proposition}
\begin{proof}	
	First note that if $a\colon \N\to \U$ is a sequence, then for every $k, n, r\in \N$ we have that
	$$
	a(n+kr)=(1+\Delta_r)^ka(n).
	$$
	Hence, by Proposition~\ref{P:NilHardyFormula} we get that the limits
		$$
	\lim_{N\to\infty}\wrE_{n\in [N]}\prod_{j=1}^k h_j(b^{a(n+rn_j)}\cdot x),\qquad  		\lim_{N\to\infty}\wrE_{n\in [N]}\  \prod_{j=1}^k h_j(b^{[a(n+rn_j)]}\cdot x)
	$$
	exist and do not depend on $r\in \N$ where $w:=|\Delta_r^da|$. Using Lemma~\ref{L:Wr}, we get the asserted statement.
	\end{proof}

\begin{proof}[Proof of Theorem~\ref{T:NilHardyIterates}]
	It follows from Proposition~\ref{P:NilHardyFormula'} (take $r=1$) that the sequences $(b^{a(n)}x)$
and $(b^{[a(n)]}x)$ admit correlations on $\bN:=([N])_{N\in\N}$;  hence these sequences have unique F-systems. Moreover, since the limits in \eqref{E:identities}   do not depend on $r$, we get by Lemma~\ref{L:sst} that these F-systems are strongly stationary.
	\end{proof}

\subsection{Strong stationarity of Hardy field iterates}

We will use the following result that  follows from Theorem~D and Theorem 4.5 in \cite{BMR20}:
\begin{theorem}[\cite{BMR20}]\label{T:BMR}
 Let $a\colon \R_+\to \R$ be a  Hardy field function such that $t^{d+\varepsilon}\prec a(t)\prec t^{d+1}$ for some $d\in \Z_+$ and $\varepsilon>0$. Then for every ergodic  system $(X,\mu,T)$, $\ell\in \N$, functions $f_1,\ldots, f_\ell\in L^\infty(\mu)$,
 and $n_1,\ldots, n_\ell\in \Z$, the following limit exists
 $$
 \lim_{N\to\infty}\E_{n\in[N]}\prod_{j=1}^\ell T^{[a(n+n_j)]}f_j
 $$
in $L^2(\mu)$.  Furthermore, for $d\in \N$ and $n_1,\ldots, n_\ell$ distinct,  if $\E(f_j|\mathcal{Z})=0$, for some $j\in\{1,\ldots, \ell\}$, where $\mathcal{Z}$ is the infinite-step nilfactor of the system $(X,\mu,T)$,  then the limit is $0$.
\end{theorem}
\begin{remark}
	Mean convergence is proved in \cite{BMR20} with  the averages 	$\lim_{N\to\infty}\wrE_{n\in [N]}$
	in place of the averages $\lim_{N\to\infty}\E_{n\in[N]}$. One gets the asserted statement by combining this result with  Lemma~\ref{L:Wr}.
\end{remark}
\begin{proposition}\label{P:sstformula}
Let $a\colon \R_+\to \R$ be a  Hardy field function such that  $t^{d+\varepsilon}\prec a(t)\prec t^{d+1}$ for some $d\in \Z_+$ and $\varepsilon>0$. Then for every system $(X,\mu,T)$, $\ell\in \N$,  functions $f_1,\ldots, f_\ell\in L^\infty(\mu)$,
and $n_1,\ldots, n_\ell\in \Z$, the following limit exists
$$
\lim_{N\to\infty}\E_{n\in[N]}\prod_{j=1}^\ell T^{[a(n+rn_j)]}f_j
$$
in $L^2(\mu)$ and is independent of $r\in \N$.
\end{proposition}
\begin{proof}
	If $d=0$ we have  $\lim_{t\to +\infty}(a(t+1)-a(t))= 0$, and since $a(t)$ is eventually monotonic,  we get for every $h\in \Z$ that  $[a(n+h)]=[a(n)]$ for a set of $n\in\N$ with density $1$.  Therefore, the result is obvious in this case.

Suppose now that $d\in \N$. A standard ergodic decomposition argument allows us to assume that the system is ergodic. Using Theorem~\ref{T:BMR} we can assume that all functions are $\mathcal{Z}$-measurable
where $\mathcal{Z}$ is the infinite-step nilfactor of the system. Using the Host-Kra theory of characteristic factors
\cite{HK05} (see also \cite[Theorem~4.2]{HK18}) and a standard  approximation
argument we can assume that the system is an ergodic nilsystem and the functions are continuous. In this case the result follows from Proposition~\ref{P:NilHardyFormula'}.
\end{proof}

\begin{proof}[Proof of Theorem~\ref{T:SstHardyIterates}]
	Let $N_k\to \infty$ be a strictly increasing sequence of integers,    $(Y,\nu,S)$ be a system, and $\mathcal{G}\subset L^\infty(\mu)$ be a countable collection of functions that is dense in $L^\infty(\mu)$ with the $L^2(\mu)$ norm.
Recall that  mean convergence of a sequence of functions implies pointwise convergence along a subsequence.
With this in mind, using the convergence result of  Theorem~\ref{T:BMR} and a diagonal argument, we get that there exists a subsequence  $(N_k')$ of $(N_k)$ such that for $\nu$-almost every $y\in Y$  and for every $g\in \mathcal{G}$  the sequence $(g(S^{[a(n)]}y))$ admits correlations on $\bN':=([N_k'])$.
Hence, for  almost every
$y\in Y$ for every $g\in \mathcal{G}$ the limit
$$
\E_{n\in\bN'}\prod_{j=1}^\ell g_j(S^{[a(n+rn_j)]}y)
$$
exists for all  $\ell, r\in \N$, $n_1, \ldots, n_\ell\in \Z$, and  $g_1,\ldots, g_\ell \in \{g,\overline{g}\}$. Since this limit coincides with the $L^2(\mu)$-limit,
Proposition~\ref{P:sstformula} implies that for almost every $y\in Y$ it is independent of $r\in \N$. Furthermore, using an approximation argument we get that a similar property holds with the set
 $\mathcal{G}$ replaced with $L^\infty(\nu)$. Using Lemma~\ref{L:sst} we get that for almost every $y\in Y$, for every $g\in L^\infty(\nu)$  the sequence $(g(S^{[a(n)]}y))$ admits correlations on $\bN'$ and the corresponding F-system is strongly stationary. This completes the proof.
	\end{proof}

 \subsection{Proof of Corollary~\ref{C:HardyIterates}}
We prove Part~$(i)$. Suppose that the conclusion fails.  Then for some  $f,g\in L^\infty(\mu)$ there exist  $\varepsilon>0$ and  $N_k\to \infty$ such that
 \begin{equation}\label{E:varepsilon}
 \norm{\E_{n\in[N_k]}\, T^nf\cdot  S^{[a(n)]}g -\E(f|\mathcal{I}_T) \cdot  \E(g|\mathcal{I}_S)}_{L^2(\mu)}\geq \varepsilon
 \end{equation}
 for every $k\in \N$ where $\mathcal{I}_T:=\{h\in L^2(\mu)\colon Th=h\}$ and  $\mathcal{I}_S$ is defined similarly. By Theorem~\ref{T:HardyIterates} there exists a subsequence $(N_k')$ of $(N_k)$ such that for almost every $x\in X$  the  sequence $(g(S^{[a(n)]}x))$ admits correlations on $\bN':=([N'_k])$ and the corresponding F-systems have trivial spectrum and their ergodic components  are isomorphic to direct products of infinite-step nilsystems and Bernoulli systems. Note also that by Proposition~\ref{P:ergodic} for almost  every $x\in X$ the sequence  $(f(T^nx))$ admits correlations on $([N])_{N\in \N}$ and the corresponding F-systems are ergodic and have zero entropy (by assumption).

  It follows from  \cite[Proposition~3.12]{FH18}  that for almost every $x\in X$ the F-system of  the sequence  $(g(S^{[a(n)]}x))$ on $\bN'$
  and the F-system of  the sequence $(f(T^nx))$ on $\bN'$ are disjoint.
   Using a standard disjointness argument we deduce from this that  for almost every $x\in X$ we have
$$
 \E_{n\in \bN'}\, f(T^nx)\cdot  g(S^{[a(n)]}x)=
 \E_{n\in\bN'}\, f(T^nx)\cdot  \E_{n\in \bN'}\,  g(S^{[a(n)]}x).
 $$
  Lastly, note that by the ergodic theorem and \cite{BKQW05}  we have for almost every $x\in X$ that
 $$
 \E_{n\in\N} \, f(T^nx)= \E(f|\mathcal{I}_T)(x), \qquad  \E_{n\in\N} \, g(S^{[a(n)]}x) = \E(g|\mathcal{I}_S)(x).
 $$
 Combining these facts, and using the bounded convergence theorem,  we get a contradiction from \eqref{E:varepsilon}, completing the proof.

 Part~$(ii)$ follows immediately form Part~$(i)$ and the estimate in  \cite[Lemma~1.6]{Chu11}.

\end{document}